\documentclass{amsart}

\RequirePackage{amsthm,amsmath,amsfonts,amssymb}
\RequirePackage[numbers]{natbib}

\usepackage{stmaryrd}
\DeclareMathAlphabet{\mathpzc}{OT1}{pzc}{m}{it}
\usepackage[inline]{enumitem}
\usepackage{url}
\usepackage{graphicx}
\usepackage{mathtools}
\usepackage{xcolor}
\usepackage{hyperref}
\usepackage{bbm,bm}
\usepackage{mathrsfs}
\usepackage{amssymb}
\usepackage{amsfonts}
\usepackage{upgreek}
\usepackage{nicefrac}
\usepackage{ifthen}
\usepackage{xargs}
\usepackage{caption}
\usepackage{subcaption}
\usepackage{aliascnt}
\usepackage{cleveref}
\usepackage{todonotes}
\numberwithin{equation}{section}

\theoremstyle{plain}
\newtheorem{theorem}{Theorem}
\crefname{theorem}{theorem}{Theorems}
\Crefname{Theorem}{Theorem}{Theorems}

\newaliascnt{lemma}{theorem}
\newtheorem{lemma}[lemma]{Lemma}
\aliascntresetthe{lemma}
\crefname{lemma}{lemma}{lemmas}
\Crefname{Lemma}{Lemma}{Lemmas}

\newaliascnt{corollary}{theorem}
\newtheorem{corollary}[corollary]{Corollary}
\aliascntresetthe{corollary}
\crefname{corollary}{corollary}{corollaries}
\Crefname{Corollary}{Corollary}{Corollaries}

\newaliascnt{proposition}{theorem}
\newtheorem{proposition}[proposition]{Proposition}
\aliascntresetthe{proposition}
\crefname{proposition}{proposition}{propositions}
\Crefname{Proposition}{Proposition}{Propositions}

\theoremstyle{definition}

\crefname{definition}{definition}{definitions}
\Crefname{definition}{Definition}{Definitions}

\newaliascnt{remark}{definition}
\newtheorem{remark}[remark]{Remark}
\aliascntresetthe{remark}
\crefname{remark}{remark}{remarks}
\Crefname{Remark}{Remark}{Remarks}

\crefname{example}{example}{examples}
\Crefname{Example}{Example}{Examples}

\Crefname{assumption}{\textbf{H}\hspace{-3pt}}{\textbf{H}\hspace{-3pt}}
\crefname{assumption}{\textbf{H}}{\textbf{H}}

\newtheorem{assumptionA}{\textbf{A}\hspace{-3pt}}
\Crefname{assumptionA}{\textbf{A}\hspace{-3pt}}{\textbf{A}\hspace{-3pt}}
\crefname{assumptionA}{\textbf{A}}{\textbf{A}}

\Crefname{assumptionG}{\textbf{G}\hspace{-3pt}}{\textbf{G}\hspace{-3pt}}
\crefname{assumptionG}{\textbf{G}}{\textbf{G}}

\Crefname{assumptionAO}{\textbf{AO}\hspace{-3pt}}{\textbf{AO}\hspace{-3pt}}
\crefname{assumptionAO}{\textbf{AO}}{\textbf{AO}}

\newtheorem{assumptionAp}{\textbf{A}\hspace{-3pt}}
\Crefname{assumptionAp}{\textbf{A}\hspace{-3pt}}{\textbf{A}\hspace{-3pt}}
\crefname{assumptionAp}{\textbf{A}}{\textbf{A}}

\newcommand{\bes}{\begin{equation*}}
\newcommand{\ees}{\end{equation*}}
\newcommand{\beas}{\begin{eqnarray}}
\newcommand{\eeas}{\end{eqnarray}}
\newcommand{\bea}{\begin{eqnarray}}
\newcommand{\eea}{\end{eqnarray}}
\newcommand{\be}{\begin{equation}}
\newcommand{\ee}{\end{equation}}
\newcommand{\bei}{\begin{itemize}}
\newcommand{\eei}{\end{itemize}}
\newcommand{\bec}{\begin{cases}}
\newcommand{\eec}{\end{cases}}
\newcommand{\ben}{\begin{enumerate}}
\newcommand{\een}{\end{enumerate}}

\newcommand{\ps}[2]{\left\langle#1,#2 \right\rangle}

\def\mcb{\mathcal{B}}  

\def\mcg{\mathcal{G}}

\def\mcp{\mathcal{P}}

\def\mcp{\mathcal{P}}

\def\msa{\mathsf{A}}

\def\frp{\mathfrak{p}}
\def\frq{\mathfrak{q}}

\def\PE{\mathbb{E}}

\def\plusinfty{+\infty}

\def\pistarT{\pi^{\star}}
\def\txts{\textstyle}

\def\kappazeta{\kappa_{\zeta}}
\def\ellzeta{\ell_{\zeta}}
\def\rmC{\mathrm{C}}
\def\alc{\mathrm{alc}}

\def\mcpalcd{\mcp_{\alc}(\rset^d)}

\def\tmcg{\tilde{\mcg}}

\def\bfW{\mathbf{W}}

\newcommand{\bbRD}{\mathbb{R}^d}
\newcommand{\N}{\mathbb{N}}

\newcommand{\cB}{\mathcal{B}}
\newcommand{\cC}{\mathcal{C}}
\newcommand{\cF}{\mathcal{F}}

\newcommand{\cG}{\mathcal{G}}

\newcommand{\cP}{\mathcal{P}}

\newcommand{\cL}{\mathcal{L}}

\newcommand{\scrH}{\mathscr{H}}

\newcommand{\bbl}{\begin{block}}
\newcommand{\ebl}{\end{block}}

\newcommand{\De}{\mathrm{d}}

\newcommand{\rmR}{\mathrm{R}}
\newcommand{\rmL}{\mathrm{L}}
\newcommand{\rme}{\mathrm{e}}

\newcommand{\Id}{\operatorname{Id}}

\newcommand{\abs}[1]{\left\vert #1 \right\vert}

\newcommandx{\Vnorm}[2][1=V]{\| #2 \|_{#1}}
\newcommandx{\norm}[2][1=]{\ifthenelse{\equal{#1}{}}{\left\Vert #2 \right\Vert}{\left\Vert #2 \right\Vert^{#1}}}
\newcommandx{\normLigne}[2][1=]{\ifthenelse{\equal{#1}{}}{\Vert #2 \Vert}{\Vert #2\Vert^{#1}}}

\def\rset{\mathbb{R}}

\def\rsets{\mathbb{R}^*}

\def\nset{\mathbb{N}}
\def\nsets{\mathbb{N}^*}

\def\ie{\textit{i.e.}}

\def\eqsp{\;}

\def\rmd{\mathrm{d}}
\def\Ent{\mathrm{Ent}}
\def\Leb{\mathrm{Leb}}
\def\varphistar{\varphi^\star}
\def\psistar{\psi^{\star}}

\def\Leb{\mathrm{Leb}}
\def\varphistar{\varphi^\star}
\def\psistar{\psi^{\star}}

\title[Exponential convergence of Sinkhorn's algorithm]{
Quantitative contraction rates for Sinkhorn's algorithm:
beyond bounded costs and  compact marginals}

\author{Giovanni Conforti}
\address{Università degli Studi di Padova}
\curraddr{Dipartimento di Matematica ``Tullio Levi-Civita'', Università degli Studi di Padova, Italia.}
\email{giovanni.conforti@math.unipd.it}
\thanks{GC acknowledges funding from the grant SPOT (ANR-20-CE40-0014).}

\author{Alain Oliviero Durmus}
\address{\'Ecole Polytechnique}
\curraddr{CMAP, {\'E}cole polytechnique, IP Paris, France.}
\email{alain.durmus@polytechnique.edu}
\thanks{AD would like
to thank the Isaac Newton Institute for Mathematical Sciences for support and hospitality during the programme
\emph{The mathematical and statistical foundation of future data-driven engineering} when work on this paper was undertaken. This work was supported by: EPSRC grant number
EP/R014604/1.}

\author{Giacomo Greco}
\address{Università degli Studi di Roma Tor Vergata}
\curraddr{RoMaDS - Dipartimento di Matematica, Università degli Studi di Roma Tor Vergata, Italia.}
\email{greco@mat.uniroma2.it}
\thanks{GG  thanks \'Ecole Polytechnique for its hospitality, where this research has been carried out, and NDNS+ for funding his visit there (NDNS/2023.004 PhD Travel Grant).  GG was partially supported by the NWO Research Project 613.009.111 "Analysis meets Stochastics: Scaling limits in complex systems", and by the PRIN project GRAFIA (CUP:
E53D23005530006). GG is currently supported by the MUR Departement of Excellence Programme 2023-2027 MatMod@Tov (CUP: E83C23000330006).}

\begin{document}

\begin{abstract}
We show non-asymptotic exponential convergence of Sinkhorn iterates to the Schrödinger potentials, solutions of the quadratic Entropic Optimal Transport problem on $\mathbb{R}^ d$. 
Our results hold under mild assumptions on the marginal inputs: in particular, we only assume that they admit an asymptotically positive log-concavity profile, covering as special cases log-concave distributions and bounded smooth perturbations of quadratic potentials. 
Up to the authors'
knowledge, these are the first results which establish exponential
convergence of Sinkhorn's algorithm in a general setting without assuming bounded cost
functions or compactly supported marginals.


\end{abstract}

\maketitle

\section{Introduction}

Given two distributions $\mu,\nu \in \mcp(\rset^d)$, we consider in this paper the problem of finding a solution to the entropic optimal transport problem: for $T \geq 0$,

\begin{equation}
\label{EOT}\tag{EOT}
\text{minimize } \int
\frac{|x-y|^2}{2}\,\De\pi(x,y)+T\,\scrH(\pi|\mu\otimes\nu) \text{
  under the constraint } \pi\in\Pi(\mu,\nu)\eqsp,
\end{equation}
where $\scrH$ denotes relative entropy (or Kullback-Leibler
divergence) and $\Pi(\mu,\nu)$ is the set of couplings of $\mu$ and
$\nu$. In \ref{EOT}, the case $T =0$ corresponds to the
Monge-Kantorovich or Optimal Transport (OT) problem and $T$ has to be
understood as a regularization parameter. Remarkably, the variational problem \ref{EOT}, whose study in statistical machine learning is motivated by its computational advantages over the Monge-Kantorovich problem, is equivalent to the Schr\"odinger
problem (SP), that is an old statistical mechanics problem whose origin can be traced back to a
question posed by Schr\"odinger \cite{Schr32} about the most likely
behaviour of a cloud of diffusive particles conditionally on the
observation of its configuration at an initial and final time. We refer to \cite{LeoSch,
  Marcel:notes} for surveys on SP and \ref{EOT}.
Both SP and \ref{EOT} stand nowadays at the crossroads of different research areas in both pure and 
applied mathematics including stochastic optimal control
\cite{chen2016relation,chen2021stochastic,mikami2006duality}, geometry
of optimal transport and functional inequalities
\cite{backhoff2020mean,conforti2018around,conforti2019second,fathi2019proof,gentil2020dynamical,gentil2020entropic,pal2020multiplicative},
abstract metric spaces
\cite{GigTam19,GigTam21,monsaingeon2020dynamical}, numerics for PDEs
and gradient flows
\cite{benamou2015iterative,BenamouEtAl19,benamou2021optimal,peyre2015entropicapprox}.

Recently, optimal transport and its entropic version have drawn
interest in statistics and machine learning \cite{peyre2019computational} and find a wealth of
applications such as inverse problems 
\cite{adler2017learning}, classification \cite{kusnerb15},
generative modeling
\cite{de2021diffusion,arjovsky2017wasserstein}, computer vision
\cite{dominitz2009texture,solomon2015convolutional} and signal
processing \cite{kolouri2017optimal}. However, in most cases, the numerical resolution of OT is challenging while  \ref{EOT}  turns to be simpler to tackle computationally for moderately small parameter
$T$ and  its solutions are used as proxies for solutions of OT \cite{cuturi2013sinkhorn,peyre2019computational}. 
This is in particular justified by  \cite{Nutz2022potentials,pooladian2021entropic} which show convergence of solutions of  \ref{EOT}  to OT as $T \to 0$.
As just mentioned,  \ref{EOT}  is simpler to solve compared to OT mainly due to Sinkhorn's algorithm also called Iterative Proportional Fitting Procedure, which can be used at the basis of numerical schemes aiming to get approximate solutions for  \ref{EOT} .

To introduce Sinkhorn's algorithm, that is the object of study of this work, we
begin by recalling that under mild assumptions on the marginals
$\mu,\nu$ (see for instance \cite[Proposition 2.2]{lagg2022gradient}),
\eqref{EOT} admits a unique minimizer $\pistarT\in\Pi(\mu,\nu)$,
referred to as the Schr\"odinger plan, and there exist two functions
$\varphi^\star \in \mathrm{L}^1(\mu)$ and
$\psi^\star \in \mathrm{L}^1(\nu)$, called Schr\"odinger
potentials, such that
\be\label{eq:sch_potentials}
\pistarT(\De x\De y)=(2\pi
T)^{-d/2}\,\exp\biggl(-\frac{|x-y|^2}{2T}-\varphi^\star(x)-\psi^\star(y)\biggr)\De
x\,\De y\eqsp.
\ee
Moreover, the couple
$(\varphi^\star,\psi^\star)$ is unique up to constant translations
$a\mapsto (\varphi^\star+a,\psi^\star -a)$.  If we suppose that the
marginals admit densities of the form
\be\label{log:marginals}
\mu(\De
x)=\exp(-U_\mu(x))\De x\,,\qquad\nu(\De x)=\exp(-U_\nu(x))\De x\eqsp,
\ee
then, imposing that a noisy transport plan of the form \eqref{eq:sch_potentials} belongs to $\Pi(\mu,\nu)$ one finds that $\varphistar,\psistar$ solve  the
Schr\"odinger system
\be\label{schrodinger:system}
   \begin{cases}
   \varphistar= U_\mu + \log P_T \exp(-\psistar)\\
   \psistar= U_\nu + \log P_T \exp(-\varphistar)\eqsp,
   \end{cases}
\ee
where $(P_t)_{t\geq 0}$ is the Markov semigroup generated by the standard $d$-dimensional Brownian motion $(B_t)_{t \geq 0}$: $P_tf(x) = \PE[f(x+B_t)]$ for any non-negative measurable function $f:\rset^d\to \rset$.

Then, starting from a given initialization $\psi^0,\varphi^{0}\colon\bbRD\to\rset$, Sinkhorn's algorithm solves \eqref{schrodinger:system} as a fixed point problem by generating two sequences of potentials $\{\varphi^n,\psi^n\}_{n \in\nset}$, called Sinkhorn potentials according to the following recursive scheme:
\be\label{eq:sinkhorn_iterate}
   \begin{cases}
   \varphi^{n+1}= U_\mu + \log P_T \exp(-\psi^{n})\\
   \psi^{n+1}= U_\nu + \log P_T \exp(-\varphi^{n+1})\eqsp.
   \end{cases}
\ee

The main goal of the present paper is to provide a new convergence analysis for Sinkhorn iterates to the Schr\"odinger potentials $\varphistar$ and $\psistar$, and hence the optimal plan $\pistarT$.
When working on discrete spaces, convergence of the algorithm is well-known and we refer to the book \cite{peyre2019computational} for an extensive overview.
However, let us mention that in the discrete setting, convergence of Sinkhorn's algorithm  dates back at least to \cite{Sinkhorn64} and \cite{SinkhornKnopp67}. More recently, \cite{Franklin89hilbert,borwein1994entropy} show that Sinkhorn iterates are equivalent to a sequence of iterates associated to an appropriate contraction in the Hilbert projective metrics, therefore proving the convergence of the algorithm boils down to studying a fixed-point problem, whose (exponential) convergence can be deduced from Birckoff's Theorem.
In the case of continuous state spaces, the use of Birckoff's Theorem for the Hilbert metrics has also been employed by \cite{chen2016hilbertmetric}, in order to establish the exponential convergence of Sinkhorn's algorithm under the condition that the state space is compact or that the cost function is bounded.  

When considering non-compact continuous state spaces and unbounded costs, results are scarcer; see  \Cref{sec:overview} for a  detailed literature overview. In addition, to
the best of our knowledge and understanding, exponential convergence of Sinkhorn iterates for  unbounded costs and marginals $\mu,\nu$ with unbounded support is still an open problem. In this paper, we address this important question by showing convergence of Sinkhorn iterates and its derivatives in the landmark example of the quadratic cost for a large class of unbounded marginal distributions and large enough values of the regularisation parameter $T$. 

More precisely, our  first main contribution is to establish exponential $\mathrm{L}^1$-convergence bounds for the gradients of Sinkhorn iterates $\{\nabla \varphi^n, \nabla \psi^n\}_{n\in\nsets}$ to the gradients of the Schrödinger potentials $\nabla \varphistar$ and $\nabla \psistar$. 
Our second main contribution is to prove quantitative exponential convergence of Sinkhorn's plans towards the Schr\"odinger plan in Wasserstein distance and, under more stringent assumptions, also in relative entropy.  In contrast to our result, previous works only established linear convergence rates in the unbounded cost setting, see \cite{ghosal2022nutz}.

\medskip

\noindent\textbf{Organisation.} The rest of the paper is organized as follows. We start \Cref{sec:ass_and_res} introducing Sinkhorn's algorithm in more detail, then we state informally our main results in \Cref{sec:main:results} where we also explain our proof strategy and our main assumptions. We conclude the introductory material with \Cref{sec:overview} which is devoted to a review of the existing literature on the convergence of Sinkhorn's algorithm. Then in \Cref{sec:gio} we analyse the (weak) convexity propagation along Sinkhorn's algorithm. That allows to state and prove our main results in the strongly convex regime in \Cref{sec:proof:strong} and in the weakly convex regime in \Cref{sec:proof:weak}. Important tools that are employed in the weak setting and which are based on the coupling by reflection technique are postponed to \Cref{new:sec}. The contractive results proven in \Cref{new:sec} are stated in full generality since they are of independent interest. 
  Finally, in \Cref{technicalities} we have postponed the proof of some technical results employed in our proofs.

\medskip
\noindent
\textbf{Notation.}
The set $\rset^d$ is endowed with the standard Euclidean metric and we denote by $\abs{\cdot}$ and $\ps{\cdot}{\cdot}$ the corresponding norm and scalar product.  For $a \in\rset$, we denote by $a_+= 0 \vee a$ and $a_{-} = 0 \vee (-a)$.  We denote by $\mcb(\rset^d)$ the Borel $\sigma$-field of $\rset^d$, by $\cP(\rset^d)$ the space of probability measures defined on $(\rset^d,\mcb(\rset^d))$ and by $\cP_2(\rset^d)$ the subset of probability measures with finite second moment. 
 For any probability measure $\mu\in\cP(\bbRD)$ we denote with $M_k(\mu) = \int \abs{x}^k \rmd \mu(x)$ its $k^{th}$ moment and define its covariance matrix as $\mathrm{Cov}(\mu)\coloneqq \int xx^{\sf T}\De \mu-(\int x \De\mu)(\int x \De\mu)^{\sf T}$. 
For two distributions $\mu_1,\mu_2$ on $(\rset^d,\mathcal{B}(\rset^d))$, 
$\Pi(\mu_1,\mu_2)$ denotes the set of couplings between $\mu_1$ and $\mu_2$, \ie , 
  $\xi\in\Pi(\mu_1,\mu_2)$ if and only if $\xi$ is a probability measure on $\bbRD\times\bbRD$ and $\xi(\msa \times \rset^d) = \mu_1(\msa)$ and $\xi(\rset^d\times \msa) = \mu_2(\msa)$.  On $\bbRD\times\bbRD$ we consider the projection operators $\mathrm{proj}_x(a,b)=a$ and $\mathrm{proj}_y(a,b)=b$. For any couple of functions $f,g\colon\bbRD\to\rset$ let us denote by $f\oplus g\colon\rset^{2d}\to\rset$ the function defined as $f\oplus g(x,y)\coloneqq f(x)+g(y)$. For $f :\rset^d\to\rset^d$ measurable and $\mu \in\mcp(\rset^d)$, we denote by $f_{\sharp} \mu$ the pushforward measure of $\mu$ by $f$: $f_{\sharp} \mu(\msa) = \mu(f^{-1}(\msa))$ for any $\msa \in \mcb(\rset^d)$.
For any function $f\colon\bbRD\to\rset$, denote by $\norm{f}_\infty\coloneqq \sup_{x\in\bbRD}|f(x)|$ its supremum norm.
Denote by $\rm\mathrm{L}^1 (\mu)$ the set of functions integrable with respect to $\mu$. For two (non-negative) measures $\lambda_1,\lambda_2$, we define the relative entropy (or Kullback Leibler divergence) between $\lambda_1$ and $\lambda_2$ as  $\scrH(\lambda_1 \mid \lambda_2)= \int_{\rset^d}  \log(\De \lambda_1/\De \lambda_2)\De \lambda_1$ if $\lambda_1 \ll \lambda_2$ and $\int_{\rset^d}  [\log(\De \lambda_1/\De \lambda_2)]_{-}\De \lambda_1 < \plusinfty$, and $\scrH(\lambda_1 \mid \lambda_2)= +\infty$ otherwise. We denote by $\Leb$ the Lebesgue measure and  define $\Ent(\lambda_1) = \int  \log (\rmd \lambda_1 / \rmd \Leb) \rmd \lambda_1$ if  $\lambda_1 \ll \Leb$ and $+\infty$ otherwise. Finally, for $p\in\{1,2\}$ we denote with $\bfW_p$ the Wasserstein distance of order $p$ which is defined for any $\lambda_1,\lambda_2\in\cP(\bbRD)$ as $\bfW_p(\lambda_1,\lambda_2)\coloneqq (\inf_{\pi\in\Pi(\lambda_1,\lambda_2)}\int|x-y|^p\De\pi(x,y))^{\nicefrac1p}$.

\section{Sinkhorn's algorithm and its geometric convergence}\label{sec:ass_and_res}

To gain a deeper understanding and build intuition on how Sinkhorn's algorithms works, it is worth recalling some basic results on the Schr\"odinger problem: this interlude allows us to naturally introduce objects and concepts that will be used extensively throughout our convergence analysis. 
The Schr\"odinger problem (SP) associated with  \ref{EOT} , is an equivalent formulation of  \ref{EOT}  inspired by Schr\"odinger's seminal work \cite{Schr32} about the behaviour of Brownian particles conditionally to some observations; see \cite[Sec. 6]{LeoSch} for a more detailed discussion on the statistical mechanics motivation behind Schr\"odinger's question.

Let  $\mu,\,\nu$ be  two probability measures on $\bbRD$ satisfying $\Ent(\mu) + \Ent(\nu) < \plusinfty$, where $\Ent$ denotes the relative entropy with respect to the Lebesgue measure. Given a horizon 
$T >0$,
SP consists in solving the optimization problem
\be\label{SP}\tag{SP}
\text{minimize }  \scrH(\pi|\rmR_{0,T}) \text{ under the constraint } \pi\in\Pi(\mu,\nu)\eqsp,
\ee
where $\scrH(\pi|\rmR_{0,T})$ is the relative entropy between the coupling $\pi$ and $\rmR_{0,T}$ defined by
\begin{equation*}
\mathrm{R}_{0,T}(\De x\De y)=(2\pi T)^{-d/2}\,\exp\biggl(-\frac{|x-y|^2}{2T}\biggr)\De x\,\De y\eqsp.
\end{equation*}
 Indeed, the identity 
\bes
\begin{split}
T\scrH(\pi|\rmR_{0,T})-&T\Ent(\mu)-T\Ent(\nu)-\frac{dT}{2}\log(2\pi T)\\
=&\int \frac{|x-y|^2}{2}\De\pi(x,y)+T\scrH(\pi|\mu\otimes\nu)\, \text{ for any}\eqsp\pi\in\Pi(\mu,\nu)
\end{split}\ees
immediately implies the equivalence between \ref{SP} and     \ref{EOT}  .

Based on the formulation \eqref{SP}  it has been pointed out \cite{benamou2015iterative} that Sinkhorn's algorithm is a special case of the Bregman's iterated projection algorithm for relative entropy. Indeed, in the current setup Bregman's iterated projection algorithm produces two sequence of plans $(\pi^{n,n},\pi^{n+1,n})_{n\in\nsets}$ starting from a positive measure $\pi^{0,0}$ according to the following recursion:
\be
\label{eq:primal_pb}
\pi^{n+1,n}\coloneqq{\arg\min}_{\Pi(\mu,\star)} \scrH(\cdot|\pi^{n,n})\,, \qquad\pi^{n+1,n+1}\coloneqq {\arg\min}_{\Pi(\star,\nu)} \scrH(\cdot|\pi^{n+1,n})\,,
\ee
where $\Pi(\mu,\star)$ (resp. $\Pi(\star,\nu)$) is the set of probability measures $\pi$ on $\mathbb{R}^{2d}$ such that the first marginal is $\mu$, \ie, $(\mathrm{proj}_x)_{\sharp}\pi = \mu$  (resp. the second marginal is $\nu$, \ie, $(\mathrm{proj}_y)_{\sharp}\pi = \nu$).
It is relatively easy to show (see \cite[Section 6]{Marcel:notes}), starting from $\pi^{0,0}(\rmd x \rmd y) \propto \exp(-\psi^{0}(y) - \varphi^{0}(x)) \rmd x \rmd y$, that the  iterates in \eqref{eq:primal_pb} are related to Sinkhorn potentials through 
\begin{align}
  \label{eq:sp:mu}
\txts \pi^{n+1,n}(\De x \De y)& \txts \propto\exp(-\nicefrac{|x-y|^2}{2T}-\varphi^{n+1}(x)-\psi^n(y)) \rmd x \rmd y \eqsp,  \\ \label{eq:sp:nu}
\txts   \pi^{n+1,n+1}(\De x\De y) & \txts \propto\exp(-\nicefrac{|x-y|^2}{2T}-\varphi^{n+1}(x)-\psi^{n+1}(y))\rmd x \rmd y \eqsp.
\end{align}
In the sequel, we will refer to the couplings $(\pi^{n,n},\pi^{n+1,n})_{n\in\nsets}$ as Sinkhorn plans.

\subsection{Main results and proof strategy}\label{sec:main:results}
In order to prove the exponential convergence of Sinkhorn's algorithm, we have investigated how the gradients of Sinkhorn's iterates behaves along the algorithm. Particularly, we have relied on a  well known result, see \cite[Proposition 2]{pooladian2021entropic} or \cite{chewi2022entropic} for instance, that links gradients (and hessians) along Hamilton-Jacobi-Bellman equations with conditional expectations, which states that
\be\label{eq:expl:grad}
\begin{aligned}
\nabla \log P_T\rme^{-h}(x) =&\, T^{-1}\int (y-x)\, \pi_T^{x,h}(\De y)\eqsp,\\
\nabla^2 \log P_T\rme^{-h}(x) =&\, -T^{-1}\operatorname{Id}+ T^{-2}\,\mathrm{Cov}(\pi_T^{x,h})\eqsp,
\end{aligned}
\ee
where 
 $(x,A)\mapsto \pi_T^{x,h}(A)$ is defined as the Markov kernel on $\bbRD\times \cB(\bbRD)$ whose transition density w.r.t. Lebesgue  measure is proportional to $(x,y) \mapsto \exp(-h(y)-|x-y|^2/(2T))$, \ie, for any $x\in\rset^d$, $\pi^{x,h}_T$ is defined through
    \be\label{eq:cond_distr}
  \pi_T^{x,h}(\De y) \propto \exp\biggl(-\frac{|y-x|^2}{2T}-h(y)\biggr) \De y \eqsp.
\ee

The proof of this technical result is given for completeness in  \Cref{sec:proof-crefprop:cov}.

 \medskip
 
This suggests that, as soon as we aim to bound the difference between gradients in $p$-norm (for $p\in\{1,\,2\}$) 
\be\label{eq:rel:norm:wass:intro}
\begin{aligned}
\|\nabla \varphi^{n+1} - \nabla\varphistar\|^p_{\rmL^p(\mu)}=\|\nabla \log P_T e^{-\psi^n} - \nabla\log P_T e^{-\psistar}\|^p_{\rmL^p(\mu)}\\
\leq T^{-p}\,\int\bfW_p^p(\pi_T^{x,\psi^n},\pi_T^{x,\psistar})\eqsp \mu(\De x)\eqsp,
\end{aligned}\ee
we should investigate the Wasserstein distance between the conditional measures $\pi_T^{x,\psi^n}$ and $\pi_T^{x,\psistar}$. One way of estimating this distance is by relying on the (weak) convexity of their (negative) log-densities and therefore on the (weak) convexity of Sinkhorn's iterates and Schr\"odinger potentials (cf. \Cref{sec:gio}). This particularly suggests the geometric nature of the assumptions we impose on the marginals, that is we require that (weak) convexity holds for their (negative) log-densities.

For readers' convenience, we firstly state our assumptions and main results in the strongly log-concave regime and then generalise these in the weak setting.

\subsubsection{Strongly log-concave case}
Let us start by considering the strongly log-concave case, that is when the marginals $\mu,\,\nu$ (and their negative log-densities $U_\mu,\,U_\nu$, defined at~\eqref{log:marginals}) satisfy
\begin{assumptionA}\label{A:log:concave:doppio}
 There exist $\alpha_\nu,\,\alpha_\mu\in(0,+\infty)$ and $\beta_\mu,\,\beta_\nu\in(0,+\infty]$ such that 
    \be
   \alpha_\nu\leq \nabla^2 U_\nu\leq \beta_\nu\quad\text{ and }\quad\alpha_\mu\leq  \nabla^2 U_\mu\leq \beta_\mu\,,
    \ee
    where the above inequalities hold as quadratic forms, and we omit writing the identity matrix multiplied to the scalars $\alpha,\,\beta$.
\end{assumptionA}
Under this assumption the convexity/concavity of Schr\"odinger potentials has been deeply investigated \cite{chewi2022entropic, conforti2022quadratic}. In \Cref{sec:gio}, in a much further general setting, we study how this property propagates along Sinkhorn's algorithm and get the same known convexity estimates. By relying on these we prove our first main result.

    \begin{theorem}[Informal]\label{thm:W2}
    Assume \Cref{A:log:concave:doppio} and suppose $\varphi_0\equiv 0$. Then, for any 
    \be\label{T0:log:concave}
T>\frac{\beta_\mu\beta_\nu-\alpha_\mu\alpha_\nu}{\sqrt{\alpha_\mu\,\beta_\mu\,\alpha_\nu\,\beta_\nu\,(\alpha_\mu+\beta_\mu)(\alpha_\nu+\beta_\nu)}}
 \ee
    Sinkhorn's algorithm converges exponentially in $\rmL^2$, in Wasserstein $\bfW_2$ distance and in relative entropy $\scrH$ as well. More precisely, if we set 
\be\label{gamma:oo:log:concave}
\begin{aligned}
\gamma^\mu_\infty\coloneqq &\,2\biggl(\alpha_\mu+\sqrt{\alpha_\mu^2+4\alpha_\mu/(T^2\beta_\nu)}\biggr)^{-1}
\\
\gamma^\nu_\infty\coloneqq &\,2\biggl(\alpha_\nu+\sqrt{\alpha_\nu^2+4\alpha_\nu/(T^2\beta_\mu)}\biggr)^{-1}\,,
\end{aligned}
\ee
then for any $ T^{-2}\gamma_{\infty}^\mu\,\gamma_{\infty}^\nu < \lambda < 1$, there exists $C \geq 0$ such that for any $n \in\nsets$, 
\be\label{eq:W2:exp:conv:general}
\begin{aligned}
\|\nabla \varphi^{n} - \nabla\varphistar\|_{\rmL^2(\mu)} +
\|\nabla \psi^{n} - \nabla\psistar\|_{\rmL^2(\nu)}  \leq&\, C \lambda^n \|\nabla \psi^0 - \nabla\psistar\|_{\rmL^2(\nu)}    \eqsp,\\
\bfW_2(\pi^{n,n},\,\pi^\star) +
\bfW_2(\pi^{n+1,n},\,\pi^\star) \leq&\, C \lambda^n \|\nabla \psi^0 - \nabla\psistar\|_{\rmL^2(\nu)}    \eqsp,\\
\scrH(\pi^\star|\pi^{n,n}) +
 \scrH(\pistarT|\pi^{n+1,n})\leq&\, C^2 \lambda^{2n} \|\nabla \psi^0 - \nabla\psistar\|^2_{\rmL^2(\nu)} \,.
\end{aligned}
\ee
\end{theorem}

The above result follows from \Cref{thm:strong:conv}, which is proven in \Cref{sec:proof:strong}.

\begin{remark}
As $\beta_\mu=\beta_\nu=+\infty$, the exponential convergence condition \eqref{T0:log:concave} simply reads as $T>(\alpha_\mu\,\alpha_\nu)^{-1/2}$.  

Furthermore, let us point out here that in the isotropic Gaussian case, \ie, when considering Gaussian marginals $\mu=\mathcal{N}(m_\mu,\alpha_\mu\Id)$ and $\nu=\mathcal{N}(m_\nu,\alpha_\nu\Id)$ then \eqref{T0:log:concave} simply reads $T>0$\, that is we have shown the exponential convergence of Sinkhorn's algorithm for any regularising parameter $T>0$.
Let us further mention that the exponential convergence can actually be proven to hold for any $T>0$ when considering general (anisotropic) Gaussian marginals, \ie, for $\mu=\mathcal{N}(m_\mu,\Sigma_\mu)$ and $\nu=\mathcal{N}(m_\nu,\Sigma_\nu)$, by improving our propagation of convexity result. This will be done in a forthcoming work.
\end{remark}

\subsubsection{Weakly log-concave setting} The previous result can be proven for a much wider class of marginals. 
To state the assumptions we impose on marginal distributions, we introduce the integrated convexity profile
   $\kappa_U : \rsets_+ \to \rset$ of a differentiable function $U :\rset^d \to \rset$ as the  function
\bes
\kappa_U(r)\coloneqq \inf\biggl\{\frac{\langle\nabla U(x)-\nabla U(y),\,x-y\rangle}{|x-y|^2}\quad\colon\quad |x-y|=r\biggr\}\,.
\ees
Likewise, for a distribution $\zeta$ represented as a Gibbs measure as $\zeta(\rmd x) \propto \exp(-U) \rmd x$, we simply write $\kappa_{\zeta}$ for  $\kappa_{U}$ and referred to this function as to the integrated log-concavity profile of $\zeta$. The function $\kappa_U$ is often employed to quantify ergodicity of stochastic differential equations whose drift field is $-\nabla U$, see \cite{eberle2016reflection}. The term integrated convexity profile we coin here is motivated by the observation that for any $U\in\cC^2$ function we have  $\kappa_U(r) \geq \alpha$ if and only if for any $x,v \in\rset^d$, $\abs{v} = 1$, $\int_0^{r} \langle \nabla^2 U(x+h v)v,v \rangle \De h \geq \alpha r$. Henceforth, the strong convexity condition $\nabla^2U\geq\alpha$ implies $\kappa_U(r)\geq \alpha$. Conversely, if $U\in\cC^2$ and  $\kappa_U(r) \geq \alpha$ for all  $r>0$, then $U$ is strongly convex and $\nabla^2U\geq\alpha$.
The integrated concavity profile of $U$ is defined in a similar way as $\ell_U= -\kappa_{-U}$, and for $\zeta$ of the form $\zeta(\rmd x) \propto \exp(-U) \rmd x$ we set $\ellzeta = \ell_{U}$.
Our main results apply when the marginals $\mu,\nu$ satisfy the following property
for $\zeta \in \{\mu,\nu\}$,
\bes
\begin{split}
\liminf_{r\rightarrow +\infty} \kappazeta(r)>0 \eqsp, \quad
\liminf_{r\rightarrow 0} r\kappazeta(r)=0 \eqsp.
\end{split}
\ees

Below we are going to give a more precise and detailed assumption on marginals. In view of that, let  us consider two sets of functions that  will appear in our conditions:
\bes
\cG\coloneqq \left\{g\in\cC^2((0,+\infty),\,\rset_+)\, :\, \begin{aligned}&\,(r\mapsto r^{1/2}\,g(r^{1/2}))\text{ is non-decreasing and concave,}\\
&\,\lim_{r\downarrow0}r\,g(r)=0 \end{aligned}\right\}
\ees
  and its subset
\bes
\tmcg\coloneqq \biggl\{g\in\cG\eqsp\text{ bounded, s.t.}\quad\lim_{r\downarrow 0}g(r)=0\,,\eqsp\exists\,g{'}(0^+)<+\infty\,,\eqsp g{'}\geq 0\eqsp \text{ and }\eqsp 2g{''}+g\,g{'}\leq 0 \biggr\}\,.
\ees 
The above classes of functions are non-empty and in particular $\tmcg$ contains the function $r \mapsto 2\tanh(r/2)$.
Though it may not appear as the most natural at first sight, these choices will become clear in light of our proofs. Indeed, the sets $\cG$ and $\tmcg$ enjoy special invariance properties (see \Cref{gio2022:thm} and  \Cref{appo:lemma33}) under the mapping 
\be 
g \mapsto -\log P_T \exp(-g) \eqsp,
\ee 
upon which the proof of the lower bounds on the integrated convexity profiles of potentials (see \Cref{sec:gio}) are built.

We say that a potential $U : \rset^d\to \rset$ is asymptotically strongly convex  if  there exist $\alpha_{U} \in \rset_+^*$ and $\tilde{g}_{U} \in \tilde \cG$ such that 
\be\label{eq:condition_a_log_concave_def}
    \kappa_{U}(r)\geq \alpha_U-r^{-1}\,\tilde g_{U}(r)
\ee
holds for all $r\geq0$. We consider the set of asymptotically strongly log-concave probability measures
\be\label{eq:asympt_log_conc_def}
\mcpalcd=\{ \zeta(\De x) = \rme^{-U}\De x: \, U \in \cC^2(\rset^d), \, \text{$U$ asymptotically strongly convex} \}\,.
\ee
Note that as soon as $\zeta(\De x) = \exp(-U(x))\De x$ there exist $\beta_U\in(0,+\infty]$ and $g\in\cG$ such that 
\begin{equation}\label{eq:asymptot_log_conc_upper_bound}
\ell_{U}(r)\leq \beta_U+r^{-1}\,{g_U}(r)\eqsp 
\end{equation}
holds for all $r\geq0$. This is trivially true as we can choose $\beta_U=+\infty$. However, if the above holds for some $\beta_U<+
\infty$, we obtain better lower bounds on the integrated convexity profile of Sinkhorn's potentials and all contraction rates appearing in our main results are better than those obtained for $\beta_U=+\infty$.

\begin{remark}[Example of weakly log-concave measures]\label{rem:g_L}
Let us remark that the class of probability measures $\mcpalcd$ contains in particular probability measures $\zeta$ associated with potentials $U$ satisfying 
\be
\label{eq:rem_u}
    \kappa_U(r)\geq\begin{cases}
        \alpha \quad&\text{ if }r>R\\
        \alpha-C_U\quad&\text{ if }r\leq R\,,
    \end{cases}
    \ee
    for  $\alpha,C_U, R>0$. Indeed,  \cite[Proposition 5.1]{conforti2022quadratic} implies that $\kappa_{\zeta}(r)\geq \alpha-r^{-1}\eqsp {\tilde g}_L(r)$ where 
 ${\tilde g}_L\in\tmcg$ is given by 
    \begin{equation}\label{eq:tanh}
      {\tilde g}_L(r)\coloneqq2\,(L)^{1/2}\,\tanh(r\,L^{1/2}/2) \eqsp\text{ with }\quad L\coloneqq \inf\{\bar L\,\colon\,R^{-1}\,{\tilde g}_{\bar L}(R)\geq C_U \}\eqsp.
    \end{equation}

    Finally, it is worth mentioning that \eqref{eq:rem_u} holds if $U$ can be expressed as the sum of a strongly convex function and a Lipschitz function with second derivative bounded from below. This allows to cover not only potentials that are strongly convex outside a compact but also double-well and multiple-wells potentials such as $|x|^4-L|x|^2$.

\end{remark}

We have now all the concepts and  notations to introduce our assumptions.
\begin{assumptionA}\label{A:entropy}
  The marginals $\mu,\nu\in\cP(\bbRD)$ with log-densities $U_\mu,\,U_\nu$ (cf. \eqref{log:marginals}) 
belong to 
 the set $\mcpalcd$ defined at \eqref{eq:asympt_log_conc_def} and have finite relative entropy with respect to the  Lebesgue measure $\Leb$, $\Ent(\mu),\,\Ent(\nu)<+\infty$.
\end{assumptionA}

Under \Cref{A:entropy}, we denote by
$\alpha_{\mu},\beta_{\mu}$, $\tilde g_{\mu}$ and $ g_{\mu}$ (resp. $\alpha_{\nu},\beta_{\nu}$, $\tilde g_\nu$ and $g_{\nu}$) the constants and functions associated with $\mu$ (resp. $\nu$) such that  \eqref{eq:condition_a_log_concave_def} and \eqref{eq:asymptot_log_conc_upper_bound} hold for $U_{\mu}$ (resp. $U_{\nu}$).  
   We prove in \Cref{sec:gio} that the previous set of assumptions imply lower bounds on the integrated convexity profile of Sinkhorn and Schr\"odinger potentials. This result is at the core of our proof strategy and generalise the findings \cite{conforti2022quadratic}, that are valid for Schr\"odinger but not Sinkhorn potentials and restricted to the case when $g_{\mu}=g_{\nu}=0$ and $\tilde{g}_{\mu},\tilde{g}_{\nu}$ are of the form \eqref{eq:tanh}, see Remark \ref{rem:gen_gio} for more explanations.

Finally notice that a special case of \Cref{A:entropy} corresponds to the strongly log-concave case: indeed \Cref{A:log:concave:doppio} clearly implies the validity of \Cref{A:entropy} with $\tilde g_\mu$, $g_\mu$, $\tilde g_\nu$ and $g_\nu$ all null.

\medskip

Under \Cref{A:entropy} we can then generalise \Cref{thm:W2} as follows.

\begin{theorem}[Informal]
  \label{theo:1}
  Assume that \Cref{A:entropy} holds for the marginals $\mu,\nu$ specified by~\eqref{log:marginals}. In addition, suppose that $\varphi^0\equiv 0$.
Then, there exist $C \geq 0$, a time parameter $T_0\in[0,+\infty)$ (explicitly given at~\eqref{eq:suff:T}, see also~\eqref{eq:suff:T:simplified} for a simplified expression) and $\lambda \in (0,1)$ such that for $n \in\nset$ and $T>T_0$,
  \be\label{informal:convergence}\begin{aligned}
     \|\nabla \varphi^{n} - \nabla\varphistar\|_{\rmL^1(\mu)}  + \|\nabla \psi^{n} - \nabla\psistar\|_{\rmL^1(\nu)}  &\leq C \lambda^n\\
     \bfW_1(\pi^{n,n},\,\pi^\star)+\bfW_1(\pi^{n+1,n},\,\pi^\star)&\leq  C \lambda^n
  \end{aligned}\ee
  where $\{(\varphi^n,\psi^n)\}_{n\in\nset}$ and
  $\{(\pi^{n+1,n},\pi^{n,n})\}_{n\in\nset}$ are defined in \eqref{eq:sinkhorn_iterate} and \eqref{eq:sp:mu}-\eqref{eq:sp:nu} respectively.
\end{theorem}

A more quantitative version of this result is given at \Cref{thm:integrated:convergence}. Let us just mention here that the expressions simplify when considering the \emph{convex outside a compact} setting.
\begin{remark}[Results for marginals satisfying \eqref{eq:rem_u}]\label{rem:g_L:2}
If for any $\frp\in\{\mu,\nu\}$ we let $\tilde g_\frp$ and $L_\frp\geq 0$  be defined as in \eqref{eq:tanh} in \Cref{rem:g_L}, \ie 
\bes
\tilde g_\frp(r)=2\,(L_\frp)^{1/2}\,\tanh(r\,L_\frp^{1/2}/2)
\ees
then Sinkhorn's algorithm converges exponentially fast (as in \Cref{theo:1}) for any 
\be\label{eq:suff:T:simplified}
    T^2>\frac{\cosh^4\biggl(\frac{2}{L_\mu^{1/2}}+\frac{4\,L_\mu^{1/2}}{\alpha_\mu}\biggr)}{\alpha_\mu+L_\mu}\eqsp\frac{\cosh^4\biggl(\frac{2}{L_\nu^{1/2}}+\frac{4\,L_\nu^{1/2}}{\alpha_\nu}\biggr)}{\alpha_\nu+L_\nu}\eqsp.
    \ee
Moreover the exponential convergence rate $\lambda$ appearing in \eqref{informal:convergence} in \Cref{theo:1} can be taken up to (excluded)
\bes
 \lambda >T^{-2}\, \gamma_\infty^\mu\gamma_\infty^\nu\eqsp,
\ees
where
\bes
\begin{aligned}
    \gamma_\infty^\mu= &\eqsp\frac{1}{\alpha_{\varphistar}+T^{-1}+L_\mu}\eqsp\cosh^4\biggl(2(L_\mu)^{1/2}\biggl(\frac{1}{L_\mu}+\frac{2}{\alpha_{\varphistar}+T^{-1}}\biggr)\biggr)\eqsp,\\
     \gamma_\infty^\nu= &\eqsp\frac{1}{\alpha_{\psistar}+T^{-1}+L_\nu}\eqsp\cosh^4\biggl(2(L_\nu)^{1/2}\biggl(\frac{1}{L_\nu}+\frac{2}{\alpha_{\psistar}+T^{-1}}\biggr)\biggr)\eqsp,
    \end{aligned}
    \ees
    with $\alpha_{\varphistar},\,\alpha_{\psistar}$ being the weak convexity profiles of the Schr\"odinger potentials (see \Cref{gio2022:thm} below).
\end{remark}

\medskip

It is worth mentioning that the convergence appearing in \Cref{theo:1} holds in $\bfW_1$ and not in relative entropy or in $\bfW_2$. This is due to the fact that our proof approach in the strongly log-concave setting relies on convexity functional inequalities and synchronous couplings. In the weakly log-concave setting this could not be straightforwardly applied and synchronous couplings should be replaced by reflection couplings  \cite{eberle2016reflection}. This explains why the convergence bounds appearing in \Cref{theo:1} appear with $\rmL^1$-norms and $\bfW_1$-distances. Nevertheless, a different strategy approach can be employed in order to prove convergence in relative entropies. Indeed by generalising the coupling by reflection approach employed in \Cref{theo:1} we are able to prove pointwise exponential convergence bounds for Sinkhorn's iterates $\varphi^n,\,\psi^n$ and then use the latter bounds for proving the exponential convergence of Sinkhorn's algorithm in (symmetric) relative entropies, though the rate of convergence $\lambda$, constant $C$ and threshold-condition $T>T_0$ clearly deteriorate. For sake of exposition these results are not presented here, but we refer the interested reader to \cite[Chapter 6]{greco2024thesis} where our approach is further generalised in order to prove the exponential convergence of the hessians along Sinkhorn's algorithm.

\medskip

   \subsubsection{Proof strategy}
For sake of clarity we sketch here our proof strategy for the case $p=1$, but the strategy for $p=2$ follows the same reasoning. 
   As we have already noticed in~\eqref{eq:rel:norm:wass:intro}, we should bound the Wasserstein distance between conditional measures since
   \bes
\begin{aligned}
\|\nabla \varphi^{n+1} - \nabla\varphistar\|_{\rmL^1(\mu)}
\leq T^{-1}\,\int\bfW_1(\pi_T^{x,\psi^n},\pi_T^{x,\psistar})\eqsp \mu(\De x)\eqsp.
\end{aligned}\ees
This can be done by combining together  the lower bounds on the integrated convexity profiles $\kappa_{\psi^n}$ (see \Cref{sec:gio}) and coupling techniques (see \Cref{new:sec}), since the conditional distribution $\pi_T^{x,h}$ (with $h\in\{\psi^n,\,\psistar\}$) can be seen as the invariant measure for the SDE 
\bes
\De Y_t=-\biggl(\frac{Y_t-x}{2T}+\frac12\nabla h(Y_t)\biggr)\De t+\De B_t\,.
\ees 
Then, we obtain in \Cref{lemma:expl:sticky} that for some $\gamma^{\nu}_n>0$, for any $x \in \mathbb{R}^d$,
\begin{equation*}
\bfW_1(\pi_T^{x,\psi^n},\pi_T^{x,\psistar})\leq \gamma^{\nu}_{n}\,\int|\nabla\psi^n-\nabla\psistar|(y)\eqsp\pi^{x,\psistar}_T(\De y) \, .
\end{equation*}
 Since $\mu(\De x)\otimes \pi^{x,\psistar}_T(\De y)=\pistarT(\De x\De y)\in\Pi(\mu,\nu)$, 
integrating this last estimate w.r.t.~$\mu$ finally yields to
\be\label{it2.1}
\int|\nabla \varphi^{n+1} - \nabla\varphistar|(x)\eqsp\mu(\De x) \leq \frac{\gamma^{\nu}_{n}}{T}\int |\nabla\psi^n-\nabla\psistar|(y)\nu(\De y) \eqsp.
\ee
Repeating the same argument but exchanging the roles of $\psi^n,\psistar$ and $\varphi^n,\varphistar$ we obtain
\be\label{it2.2}
\int|\nabla \psi^{n} - \nabla\psistar|(y)\eqsp\nu(\De y) \leq \frac{\gamma^{\mu}_{n-1}}{T}\int |\nabla\varphi^n-\nabla\varphistar|(x)\mu(\De x)
\ee
for some $\gamma^{\mu}_{n-1}>0$. Combining \eqref{it2.2} with \eqref{it2.1} yields
\be\label{appo:sabato31}
\|\nabla \varphi^{n+1} - \nabla\varphistar\|_{\rmL^1(\mu)}
\leq \frac{\gamma^{\nu}_{n}\gamma^{\mu}_{n-1}}{T^2} \|\nabla \varphi^{n} - \nabla\varphistar\|_{\rmL^1(\mu)}
\ee
which is a contraction provided $T^{-2}\gamma^{\nu}_{n}\gamma^{\mu}_{n-1}<1$. Therefore, provided this condition holds for $n$ large enough, iteratively applying~\eqref{appo:sabato31} allows us to establish the exponential convergence of the algorithm.

\subsection{Literature review}\label{sec:overview}

As already emphasised in the introduction, when the marginal distributions $\mu,\nu$ are discrete probability measures, exponential convergence rates for Sinkhorn's algorithm are well-known \cite{Franklin89hilbert,borwein1994entropy}.
In addition, its exponential convergence has also been established when dealing with  compact state spaces or bounded cost functions \cite{chen2016hilbertmetric,deligiannidis2021quantitative,berman2020sinkhorn}.

 When considering non compact spaces with possibly unbounded costs and unbounded marginals, the techniques developed to handle the discrete setting cannot apply as such and new ideas have emerged.  When it comes to results that allow for unbounded costs,  \cite{ruschendorf1995convergence} shows qualitative convergences of iterates in relative entropy and total variation for Sinkhorn plans. More recently, \cite{Marcel:stab:potentials} establishes qualitative convergence both on the primal and dual sides under mild assumptions. Concerning convergence rates,  \cite{leger2021gradient} gives an interpretation of Sinkhorn's algorithm as a block coordinate descent on the dual problem and obtains convergence of marginal distributions in relative entropy at a linear rate $n^{-1}$ under minimal assumptions. \cite{eckstein2021quantitative}  derives polynomial rates of convergence in Wasserstein distance assuming, among other things, that marginals admit exponential moments. Lastly,  \cite{ghosal2022nutz} improves existent polynomial convergence rates for optimal plans with respect to a symmetric relative entropy.

In contrast to these existing works, the main contribution of this paper is to establish exponential convergence bounds for the gradients of Sinkhorn iterates as well as for the optimal plans. To the best of our knowledge these findings are new both in their dual and primal formulation in that they represent the first exponential convergence results that holds for unbounded costs and marginals. They also are among the very few results that yield convergence of derivatives of potentials. On this subject, let us mention the recent work by the authors of the present paper and collaborators \cite{greco2023SinkhornTorus} where the state space is the $d$-dimensional torus and therefore deals with bounded cost functions, and the thesis \cite{greco2024thesis} where these results are generalised to the Euclidean unbounded space.  As  \cite{greco2024thesis,greco2023SinkhornTorus}, our proofs are mainly probabilistic, and differ from other proposed methodologies in that they rely on one-sided integrated semiconvexity estimates for potentials along Sinkhorn iterates. These estimates are by themselves a new result, that has potentially several further implications. Though the current approach and the one we proposed in \cite{greco2024thesis,greco2023SinkhornTorus} are both inspired by coupling methods and stochastic control, there is a fundamental difference. In \cite{greco2024thesis,greco2023SinkhornTorus} exponential convergence is achieved through Lipschitz estimates on potentials, and in the unbounded setting in \cite{greco2024thesis} both marginals are log-Lipschitz perturbation of the same log-concave reference measure.
In the current setup, we make assumptions on the integrated log-concavity profile of the marginals; these assumption are of geometric nature and not perturbative.  

We conclude this section reporting on results available for  multimarginal entropic optimal transport problem and the natural extension  of Sinkhorn's algorithm in this setting. For bounded costs and marginals, (or equivalently compact spaces)  \cite{carlier2020differential} shows well-posedness of the Schr\"odinger system and smooth dependence of Schr\"odinger potentials on the marginal inputs.  \cite{marinogerolin2020} managed to show qualitative convergence of Sinkhorn iterates to Schr\"odinger potentials in $\mathrm{L}^p$-norms using tools from   calculus of variations: their results require bounded costs but apply to multimarginal problems. These results have been subsequently improved by Carlier \cite{Carlier22multisink} who establishes exponential convergence.

Lastly, we should also mention that, after \cite{greco2023SinkhornTorus} and the results presented here appeared, very recently \cite{eckstein2023hilberts} has managed to extend the approach based on Hilbert's projective metric for unbounded settings in the general EOT problem setting, solely for marginals satisfying a light-tail condition. Though their result applies to general EOT problems for any regularising parameter $\varepsilon>0$, still in the quadratic cost setting they can't cover marginals satisfying \Cref{A:log:concave:doppio}, leaving out the landmark example of quadratic cost with Gaussian marginals (setting that on the contrary we cover for any regularising parameter $T>0$ here). Moreover, dealing with general costs prevents \cite{eckstein2023hilberts} from getting convergence results for the gradients of Sinkhorn's iterates, which is of particular interest since the gradients of Schr\"odinger potentials provide proxies for the Optimal Transport map (c.f.~\cite{lagg2022gradient}).

While this paper was under review, several new contributions has tackled the convergence rates of Sinkhorn’s algorithm. Particularly, in \cite{chizat2024sharper} the authors have shown in the bounded setting how geometric properties of marginals improve the convergence rates' dependence in $T$, from exponential to polynomial. The same polynomial dependence has been shown in \cite{chiarini2024semiconcavity} by the first and third authors of this paper (and coauthors) for a broader class of \ref{EOT} problems (including the landmark example of quadratic cost with unbounded marginals).  Lastly, in \cite{greco2025hessian} similar geometric ideas were developed in order to prove second order exponential convergence  results for Sinkhorn iterates. We refer the reader to \cite{chiarini2024semiconcavity, greco2025hessian} for more extensive comparison.

\section{Integrated semiconvexity propagation along Sinkhorn}\label{sec:gio}

In this section we establish lower bounds on the integrated convexity profile of Sinkhorn potentials. Before proceeding further, let us notice a few useful properties satisfied from the elements of $\tmcg$ and that naturally follow from its definition.
\begin{remark}\label{remark:propG}
The properties prescribed in the definition of $\tmcg$ in particular imply that its elements are concave functions on $(0,+\infty)$ and that $\sup_{r>0}g(r)/r<+\infty$ (or equivalently, that there exists some $G>0$ such that $g(r)\leq G\,r$, uniformly in $r>0$). Indeed, any $g\in\tmcg$ is clearly non-negative and non-decreasing, which combined with the differential inequality  $2g{''}+g\,g{'}\leq 0$, implies its concavity. Finally,  since $g(0^+)=0$ and $g\in\tmcg$ is bounded twice differentiable in $(0,+\infty)$ and $g{'}(0^+)$ exists and is finite,  we may also conclude that $\sup_{r>0}g(r)/r<+\infty$.
\end{remark}

Moreover, since in the following results the elements of $\tmcg$ will appear, let us rewrite here our standing assumption  by making explicit what means that the marginals are in  $\mcpalcd$. Henceforth let us rewrite \Cref{A:entropy} as
 \setcounter{assumptionAp}{1}
 \begin{assumptionAp}\label{A:kappatilde}
The two distributions $\mu,\nu\in\cP(\bbRD)$ specified by \eqref{log:marginals}  have finite relative entropy with respect to the  Lebesgue measure $\Leb$, $\Ent(\mu),\,\Ent(\nu)<+\infty$. Further assume that
    \begin{enumerate}
        \item[(i)] There exist $\alpha_\nu\in(0,+\infty)$ and $\beta_\mu\in(0,+\infty]$ such that 
    \bes
    \kappa_{U_\nu}(r)\geq \alpha_\nu-r^{-1}\,\tilde{g}_\nu(r)\quad\text{ and }\quad \ell_{U_\mu}(r)\leq \beta_\mu+r^{-1}\,g_\mu(r)\eqsp,
    \ees
    with $g_\mu\in\cG$ and $\tilde{g}_\nu\in\tmcg$;
    \item[(ii)] There exist $\alpha_\mu\in(0,+\infty)$ and $\beta_\nu\in(0,+\infty]$ such that 
    \bes
    \kappa_{U_\mu}(r)\geq \alpha_\mu-r^{-1}\,\tilde{g}_\mu(r)\quad\text{ and }\quad \ell_{U_\nu}(r)\leq \beta_\nu+r^{-1}\,g_\nu(r)\eqsp,
    \ees
     with $g_\nu\in\cG$ and $\tilde{g}_\mu\in\tmcg$.
    \end{enumerate}
\end{assumptionAp}

We have decided to split \Cref{A:entropy} here in two points because in this section our statements hold by assuming either one of the two, and whenever we write \textit{“assume \Cref{A:kappatilde}-$(i)$"} we mean that we assume the finite entropy condition for both marginals and the point $(i)$ (similarly for point $(ii)$). Lastly, we would like to mention here that what follows in this section holds under a weaker version of \Cref{A:kappatilde}, where the class $\tmcg$ is replaced with the broader class of function $g\in\cG$ bounded and such that $2g{''}+g\,g{'}\leq 0$. However, for sake of exposition, we state our main result under \Cref{A:kappatilde} and stick with the function class $\tmcg$.

 Let us introduce for any fixed $\beta>0$ and any $g\in\cG$ and $\tilde{g}\in\tmcg$, the following functions for $\alpha,\,s,\,u\geq 0$
    \be\label{def:F:G}
    \begin{aligned}
    F_\beta^{g,\tilde{g}}(\alpha,s)=&\,\beta \, s+\frac{s}{T(1+T\alpha)}+s^{1/2}\,g(s^{1/2})+ \biggl(\frac{s^{1/2}}{1+T\alpha}\biggr)\,\tilde{g}\biggl(\frac{s^{1/2}}{1+T\alpha}\biggr)\eqsp,\\
    G_\beta^{g,\tilde{g}}(\alpha,u)=&\,\inf\{s\geq 0\eqsp\colon\eqsp F_\beta^{g,\tilde{g}}(\alpha,s)\geq u\}\eqsp,
\end{aligned}
\ee  
with the convention $ G_\beta^{g,\tilde{g}}(\alpha,u)\equiv0$ whenever $\beta=+\infty$.

Then, the main result of this section can be stated as follows.
\begin{theorem}\label{gio2022:thm}
If \Cref{A:kappatilde}-$(i)$ holds and if  
\be\label{kappa:psi0}
\kappa_{\psi^0}(r)\geq \alpha_\nu-T^{-1}-r^{-1}\tilde{g}_\nu(r) \eqsp,
\ee
then there exists a monotone increasing sequence $(\alpha_{\nu,n})_{n\in\N}\subseteq (\alpha_\nu-T^{-1},\alpha_\nu-T^{-1}+(\beta_\mu\,T^2)^{-1}]$ such that for any $n\geq 1$ and $r>0$ it holds
\be\label{eq_bound_gio_teo:psi}
\ell_{\varphi^n}(r) \leq r^{-2}\,F_{\beta_\mu}^{g_\mu,\tilde{g}_\nu}(\alpha_{\nu,n},r^2)-T^{-1}\quad\text{ and }\quad\kappa_{\psi^n}(r)\geq \alpha_{\nu,n}-r^{-1}\tilde{g}_\nu(r)\eqsp,
\ee
with $F_{\beta_\mu}^{g_\mu,\tilde{g}_\nu}$ defined as in \eqref{def:F:G}. Moreover, the sequence can be explicitly built by setting
\be\label{def:sistem:alpha:psi}
\begin{cases}
\alpha_{\nu,0}\coloneqq\alpha_\nu-T^{-1}\eqsp,\\
\alpha_{\nu,n+1}\coloneqq \alpha_\nu-T^{-1}+\frac{G_{\beta_\mu}^{g_\mu,\tilde{g}_\nu}(\alpha_{\nu,n},2)}{2\eqsp T^2} \eqsp,\quad n\in\N\eqsp,
\end{cases}
\ee
 $G_{\beta_\mu}^{g_\mu,\tilde{g}_\nu}$ given in \eqref{def:F:G}.
Finally, $(\alpha_{\nu,n})_{n\in\N}$ converges to $\alpha_{\psistar}\in(\alpha_\nu-T^{-1},\alpha_\nu-T^{-1}+(\beta_\mu\,T^2)^{-1}]$, fixed point solutions of \eqref{def:sistem:alpha:psi} and for any $r >0$,
\be\label{eq_bound_gio_teo:psi:limit}
\ell_{\varphistar}(r) \leq r^{-2}\,F_{\beta_\mu}^{g_\mu,\tilde{g}_\nu}(\alpha_{\psistar},r^2)-T^{-1}\quad\text{ and }\quad\kappa_{\psistar}(r) > \alpha_{\psistar}-r^{-1}\tilde{g}_\nu(r)\eqsp,
\ee
where $\varphistar$ and $\psistar$ are the Schr\"odinger potentials introduced in \eqref{eq:sch_potentials}.

Similarly, under \Cref{A:kappatilde}-$(ii)$ there exists a monotone increasing sequence $(\alpha_{\mu,n})_{n\in\N}\subseteq(\alpha_\mu-T^{-1},\alpha_\mu-T^{-1}+(\beta_\nu\,T^2)^{-1}]$ such that for any $n\geq 1$ and $r>0$ it holds
\be\label{eq_bound_gio_teo:varphi}
\ell_{\psi^n}(r) \leq r^{-2}\,F_{\beta_\nu}^{g_\nu,\tilde{g}_\mu}(\alpha_{\mu,n},r^2)-T^{-1}\quad\text{ and }\quad\kappa_{\varphi^n}(r)\geq \alpha_{\mu,n}-r^{-1}\tilde{g}_\mu(r)\eqsp,
\ee
with $F_{\beta_\nu}^{g_\nu,\tilde{g}_\mu}$ defined as in \eqref{def:F:G} and
\be\label{def:sistem:alpha:varphi}
\begin{cases}
\alpha_{\mu,1}\coloneqq\alpha_\mu-T^{-1}\eqsp,\\
\alpha_{\mu,n+1}\coloneqq \alpha_\mu-T^{-1}+\frac{G_{\beta_\nu}^{g_\nu,\tilde{g}_\mu}(\alpha_{\mu,n},2)}{2\eqsp T^2} \eqsp,\quad n\in\N\eqsp,
\end{cases}
\ee
with $G_{\beta_\nu}^{g_\nu,\tilde{g}_\mu}$ defined as in \eqref{def:F:G}.
Finally, $(\alpha_{\mu,n})_{n\in\N}$ converges to $\alpha_{\varphistar}\in(\alpha_\mu-T^{-1},\alpha_\mu-T^{-1}+(\beta_\nu\,T^2)^{-1}]$, fixed point solutions of \eqref{def:sistem:alpha:varphi} and for any $r >0$,
\begin{equation}\label{eq_bound_gio_teo:varphi:limit}
\ell_{\psistar}(r) \leq r^{-2}\,F_{\beta_\nu}^{g_\nu,\tilde{g}_\mu}(\alpha_{\varphistar},r^2)-T^{-1}\quad\text{ and }\quad\kappa_{\varphistar}(r) > \alpha_{\varphistar}-r^{-1}\tilde{g}_\mu(r)\eqsp.
\end{equation}
\end{theorem}

\begin{remark}\label{rem:gen_gio}
The above result is an extension of \cite[Theorem 1.2]{conforti2022quadratic}
    where the author just provides   the limit-bounds 
    \eqref{eq_bound_gio_teo:psi:limit} and \eqref{eq_bound_gio_teo:varphi:limit} in the case when $g^{\mu}=g^{\nu}\equiv0$ and $\tilde{g}_\mu,\tilde{g}_{\nu}$ take the form \eqref{eq:tanh}. In the above result we show that the iterative proof given there can be actually employed when proving the estimates \eqref{eq_bound_gio_teo:psi} and \eqref{eq_bound_gio_teo:varphi} along Sinkhorn's algorithm. 
    
    Let us also mention that our result encompasses \cite[Theorem 4]{chewi2022entropic} when considering 
    $\tilde{g}_\nu\equiv 0$ and $g_\mu\equiv0$ in \Cref{A:kappatilde}-$(i)$.
\end{remark}

We provide a proof of the above theorem at the end of this section. Let us mention here that we will show in \Cref{initial:assumption} that Assumption \eqref{kappa:psi0} on $\psi^0$ can be essentially dropped; here we simply observe that it is met for a regular enough initial condition, e.g., for  $\psi^0=U_\nu$. 

Let us also point out that the above theorem guarantees existence and uniqueness for the strong solution of the SDE
\be\label{again:SDE}
\De Y_t=-\biggl(\frac{Y_t-x}{2T}+\frac12\nabla h(Y_t)\biggr)\De t+\De B_t\,.
\ee
for any choice of $h=\psistar,\,\psi^n,\,\varphistar,\,\varphi^n$. Indeed from \Cref{gio2022:thm} we immediately deduce 
\begin{corollary}\label{cor:lyap}
    Under the assumptions of \Cref{gio2022:thm}, for any even $p\geq 2$  the potential $V_p(y)=1+|y|^p$ is a Lyapunov function for \eqref{again:SDE} with  $h\in\{\psistar,\,\psi^n,\,\varphistar,\,\varphi^n\}$. As a consequence, existence and uniqueness of strong solutions holds for these SDEs.
\end{corollary}
\begin{proof}
We only focus on $h=\psistar$ since the other cases are similar. As a direct consequence of \Cref{gio2022:thm} we know that  
\bes
\left\langle \nabla\psistar(y),\,\frac{y}{|y|}\right\rangle\geq \alpha_{\psistar}\,|y|-\tilde{g}_\nu(|y|)-|\nabla \psistar(0)|\eqsp,
\ees
Then for any fixed  $x\in\bbRD$, since $\tilde{g}_\nu\in\tmcg$ is bounded, it holds for any $y\in\bbRD$
    \bes\begin{split}
-\frac12\,\left\langle T^{-1}(y-x)+\nabla \psistar(y),\,y\right\rangle\leq -\frac{\alpha_{\psistar}+T^{-1}}{2}\,|y|^2+\frac{\|\tilde{g}_\nu\|_{\infty}+|\nabla \psistar(0)|+T^{-1}|x|}{2}|y|\eqsp,
\end{split}
\ees
and hence there exist $A,B>0$ such that uniformly in $y\in\bbRD$ it holds
\bes
-\frac12\,\left\langle T^{-1}(y-x)+\nabla \psistar(y),\,y\right\rangle\leq B-A\,|y|^2\eqsp.
\ees
At this stage, \cite[Lemma 4.2]{mattingly2002} guarantees that for any even $p\geq 2$ the potential $V_p(y)=1+|y|^p$ is a Lyapunov function for the diffusion \eqref{again:SDE}.

More precisely it holds a geometric drift condition, \ie, for any  $A_{\psistar}\in (0,p\gamma)$ there exists a finite constant $B_{\psistar}=B_{\psistar}(A_{\psistar},p)$ such that for any $y\in\bbRD$
\be\label{geom:drif:h}
\cL_{\psistar} V_p(y)\leq -A_{\psistar} V_p(y)+B_{\psistar}\eqsp,
\ee
where above $\cL_{\psistar}\coloneqq  \Delta/2-\frac12\langle T^{-1}(y-x)+\nabla \psistar(y),\,\nabla\rangle$ denotes the generator associated to the SDE \eqref{again:SDE}.
Finally, existence and uniqueness of strong solutions of \eqref{again:SDE} follows from  \cite[Theorem 2.1]{tweedie96} (see also \cite[Section 2]{meyn1993tweedie}).
\end{proof}

The proof of \Cref{gio2022:thm} will be based on  a propagation of integrated-convexity along Hamilton-Jacobi-Bellman (HJB) equations observed in \cite{conforti2022quadratic}, based on coupling by reflection techniques, which reads as follows
\begin{theorem}[Theorem 2.1 in \cite{conforti2022quadratic}]\label{thm:HJB}
    For any fixed function $\tilde{g}\in\tmcg$, consider the class of functions
    \bes
    \cF_{\tilde{g}}\coloneqq\{h\in\rmC^1(\bbRD)\eqsp\colon\eqsp\kappa_h(r)\geq- r^{-1}\,\tilde{g}(r)\quad\forall\,r>0\}\eqsp.
    \ees
    Then, the class $\cF_{\tilde{g}}$ is stable under the action of the HJB flow, \ie,
    \bes
    h\in\cF_{\tilde{g}}\eqsp\Rightarrow\eqsp -\log P_{T-t}\exp(-h)\in\cF_{\tilde{g}}\quad\forall \eqsp 0\leq t\leq T\eqsp.
    \ees
\end{theorem}
We omit the proof of the above result since it runs exactly as stated in \cite{conforti2022quadratic}. There it is proven when $\tilde{g}$ is of the form \eqref{eq:tanh}. However same proofs allow to reach the conclusion for any function $\tilde{g}\in\tmcg$
since it only requires $\tilde{g}$ to satisfy the differential inequality 
\bes
2(\tilde{g}){''}+\tilde{g}\,(\tilde{g}){'}\leq 0\eqsp,
\ees
which is an equality in the special case considered there $\tilde{g}(r)=\tanh(r)$.

As a first consequence of the previous theorem we may immediately deduce the following  integrated propagation \emph{convexity-to-concavity} result.

\begin{lemma}\label{appo:lemma31}
  If \Cref{A:kappatilde}-$(i)$ holds true, $\alpha_{\nu,n}>-T^{-1}$ and if for any $r>0$
\be\label{another:assumption}
\kappa_{\psi^n}(r)\geq \alpha_{\nu,n}-r^{-1}\,\tilde{g}_\nu(r)\eqsp,
\ee
    then 
    \bes
    \ell_{\varphi^{n+1}}(r)\leq \beta_\mu+g_\mu(r)-\frac{\alpha_{\nu,n}}{1+T\alpha_{\nu,n}}+\frac{\tilde{g}_\nu\biggl(\frac{r}{1+T\alpha_{\nu,n}}\biggr)}{r(1+T\alpha_{\nu,n})}=-T^{-1}+r^{-2}\,F_{\beta_\mu}^{g_\mu,\tilde{g}_\nu}(\alpha_{\nu,n},r^2)\eqsp.
    \ees

    Similarly  if \Cref{A:kappatilde}-$(ii)$ holds, $\alpha_{\mu,n}>-T^{-1}$ and if for any $r>0$
    \bes
\kappa_{\varphi^n}(r)\geq \alpha_{\mu,n}-r^{-1}\,\tilde{g}_\mu(r)\eqsp,
\ees
    then 
     \bes
    \ell_{\psi^{n}}(r)\leq \beta_\nu+r^{-1}\,g_\nu(r)-\frac{\alpha_{\mu,n}}{1+T\alpha_{\mu,n}}+\frac{\tilde{g}_\mu\biggl(\frac{r}{1+T\alpha_{\mu,n}}\biggr)}{r(1+T\alpha_{\mu,n})}=-T^{-1}+r^{-2}\,F_{\beta_\nu}^{g_\nu,\tilde{g}_\mu}(\alpha_{\mu,n},r^2)\eqsp.
    \ees
\end{lemma}
\begin{proof}
    Let us firstly notice that our assumption \eqref{another:assumption} is equivalent to stating that
    \bes
    \bar{\psi}^n\coloneqq \psi^n-\frac{\alpha_{\nu,n}}{2}\,|\cdot|^2\,\in \cF_{\tilde{g}_\nu}\eqsp,
    \ees
    and therefore \Cref{thm:HJB} implies that $-\log P_T\exp(-\bar{\psi}^n)\in\cF_{\tilde{g}_\nu}$. 
    Moreover, since $\varphi^{n+1}$ is defined via \eqref{eq:sinkhorn_iterate}, let us preliminary notice that it can be written in terms of $-\log P_T\exp(-\bar{\psi}^n)$. Indeed, by completing the square we immediately see that
    \bes
    \begin{split}
        -\log P_T&\exp(-\psi^n)(x)-\frac{d}{2}\log(2\pi T)=-\log\int \exp\biggl(-\frac{|x-y|^2}{2T}-\frac{\alpha_{\nu,n}}{2}|y|^2-\bar{\psi}^n(y)\biggr)\eqsp\De y\\
        =&\frac{\alpha_{\nu,n}\,|x|^2}{2(1+T\alpha_{\nu,n})}-\log\int \exp\biggl(-\frac{1+T\alpha_{\nu,n}}{2T}\,|y-(1+T\alpha_{\nu,n})^{-1}x|^2-\bar{\psi}^n(y)\biggr)\eqsp\De y\\
        =&\frac{\alpha_{\nu,n}\,|x|^2}{2(1+T\alpha_{\nu,n})}-\log P_{T/(1+T\alpha_{\nu,n})}\exp(-\bar{\psi}^n)((1+T\alpha_{\nu,n})^{-1}x)-\frac{d}{2}\log\frac{2\pi T}{1+T\alpha_{\nu,n}}\,.
    \end{split}
    \ees
    Therefore we have
    \bes
    \begin{split}
    \nabla\varphi^{n+1}(x)=&\,\nabla U_\mu(x)+\nabla\log P_T e^{-\psi^n}(x)\\
    =&\nabla U_\mu(x)-\frac{\alpha_{\nu,n}\,x}{1+T\alpha_{\nu,n}}-\frac1{1+T\alpha_{\nu,n}}\,(\nabla h) ((1+T\alpha_{\nu,n})^{-1}x)\eqsp,
    \end{split}
    \ees
    where we have set $h(y)\coloneqq -\log P_{T/(1+T\alpha_{\nu,n})}\exp(-\bar{\psi}^n)(y) $ for notations' sake.
     Now, if we fix $r>0$ and take $x, y\in\bbRD$ with $|x-y|=r$, we clearly have
     \bes
     \begin{aligned}
     \langle\nabla \varphi^{n+1}(x)-\nabla \varphi^{n+1}(y),\,x-y\rangle=\langle\nabla U_\mu(x)-\nabla U_\mu(y),\,x-y\rangle-\frac{\alpha_{\nu,n}\,r^2}{1+T\alpha_{\nu,n}}\\
     -\frac{\langle(\nabla h)((1+T\alpha_{\nu,n})^{-1}x)-(\nabla h)((1+T\alpha_{\nu,n})^{-1}y),\,x-y\rangle}{1+T\alpha_{\nu,n}}\\
     \leq \beta_\mu\,r^2+r\,g_\mu(r)-\frac{\alpha_{\nu,n}\,r^2}{1+T\alpha_{\nu,n}} +\frac{r}{1+T\alpha_{\nu,n}} \tilde{g}_\nu((1+T\alpha_{\nu,n})^{-1}r)
     \end{aligned}\ees
     where the last step follows from \Cref{A:kappatilde}-$(i)$ and since $h\coloneqq -\log P_{T/(1+T\alpha_{\nu,n})}\exp(-\bar{\psi}^n)\in\cF_{\tilde{g}_\nu}$.
     Since $r>0$ and $x,\,y\in\bbRD$ were arbitrary, this concludes the proof of the first part of the statement. 
    The second part of the statement follows the same lines.
\end{proof}

\begin{lemma}\label{appo:lemma2}
    Fix $\beta\in(0,+\infty]$ and two functions $g\in\cG$ and $\tilde{g}\in\tmcg$. 
Then the following properties hold true.
    \begin{enumerate}[wide, labelwidth=!, labelindent=0pt]
        \item For any $\alpha>-T^{-1}$ the function $s\mapsto F_{\beta}^{g,\tilde{g}}(\alpha,s)$ is concave and increasing on $[0,+\infty)$.
        \item The function $\alpha\mapsto G_{\beta}^{g,\tilde{g}}(\alpha,2)$ is positive and non-decreasing over $(-T^{-1},+\infty)$.
        \item For any given $a_0>0$, the fixed-point problem
        \be\label{general:fixed:point}
        \alpha=a_0-T^{-1}+\frac{G_\beta^{g,\tilde{g}}(\alpha,2)}{2\eqsp T^2} 
        \ee
        admits at least one solution on $(a_0-T^{-1},\,a_0-T^{-1}+(\beta\,T^2)^{-1}]$ and, as soon as $\beta<+\infty$, $a_0-T^{-1}$ does not belong to the closure of the set of solutions of \eqref{general:fixed:point}.
    \end{enumerate}
\end{lemma}
\begin{proof}\,

\begin{enumerate}[wide, labelwidth=!, labelindent=0pt]
    \item Since   $\alpha>-T^{-1}$ and $r\mapsto r\,\tilde{g}(r)$ and $r\mapsto r\,g(r)$ are non-decreasing, we see that $r\mapsto \sqrt{r}\,g(\sqrt{r})$ and  $r\mapsto \nicefrac{\sqrt{r}}{1+T\alpha}\,\tilde{g}(\nicefrac{\sqrt{r}}{1+T\alpha})$  are too. This further implies that $s\mapsto F_\beta^{g,\tilde{g}}(\alpha,s)$ is an increasing function since we have
    \bes
    \frac{\De}{\De s}F_\beta^{g,\tilde{g}}(\alpha,s)\geq \beta+\frac1{T(1+T\alpha)} >0\,.
    \ees
    The concavity of $s\mapsto F_\beta^{g,\tilde{g}}(\alpha,s)$ is straightforward since sum of linear and (non-negative scalar multiples of) concave functions (since $g,\,\tilde{g}\in\cG$).

\item The proof is by contradiction.  Notice that $G_\beta^{g,\tilde{g}}(\cdot,2)$ is continuous on $(-T^{-1},+\infty)$ and assume that is not a positive function, which implies that there exists some $\alpha>-T^{-1}$ such that $G_\beta^{g,\tilde{g}}(\alpha,2)=0$ and hence by definition that there exists a sequence $(s_n)_{n\in\N}$ converging to zero and such that $F_\beta^{g,\tilde{g}}(\alpha,s)\geq 2$, which is clearly impossible since $\lim_{s\downarrow 0}F_{\beta}^{g,\tilde{g}}(\alpha,s)=0$. Hence $G_\beta^{g,\tilde{g}}(\cdot,2)$ is a positive function. The monotonicity of  $G^{g_\mu,\tilde{g}_\nu}_{\beta_\mu}(\cdot,2)$ follows from the fact that $F_{\beta}^{g,\tilde{g}}(\alpha,s)$ is increasing in $s$ and decreasing in $\alpha\in(-T^{-1},+\infty)$, which implies for any $\alpha'\geq\alpha$ and $u\geq 0$
\bes
\{s\eqsp\colon\eqsp F_{\beta}^{g,\tilde{g}}(\alpha',s)\geq u\}\subseteq \{s\eqsp\colon\eqsp F_\beta^{g,\tilde{g}}(\alpha,s)\geq u\}\eqsp.
\ees

\item 
Consider the map associated to the fixed-point problem \eqref{general:fixed:point}, \ie, the continuous function $H\colon [a_0-T^{-1},+\infty)\to\mathbb{R}$ defined for $\alpha\in(-T^{-1},+\infty)$ as
\bes
H(\alpha)\coloneqq \alpha- a_0+T^{-1}-\frac{G_\beta^{g,\tilde{g}}(\alpha,2)}{2\eqsp T^2} \eqsp.
\ees
Let us now prove that
\be\label{detail}
H(a_0-T^{-1})<0\quad\text{ and }\quad \lim_{\alpha\to+\infty} H(\alpha)=+\infty\eqsp.
\ee
The first inequality follows from the positivity of  $G_\beta^{g,\tilde{g}}(a_0-T^{-1},2)$, which has already been established in part $2$.
In order to prove the second statement it is enough noticing that $G_\beta^{g,\tilde{g}}(\cdot,2)$ is bounded, which is immediate since $g,\,\tilde{g}\geq 0$ implies for any $\alpha>-T^{-1}$ and $s> 0$ that $F_\beta^{g,\tilde{g}}(\alpha,s)\geq \beta\,s$ and hence that
\be\label{G:bounded}
 G_\beta^{g,\tilde{g}}(\alpha,2)\leq 2/\beta\eqsp.
\ee
From \eqref{detail} and the continuity of $G_\beta^{g,\tilde{g}}(\cdot,2)$ we finally deduce the existence of some $\bar \alpha\in(a_0-T^{-1},+\infty)$ such that
 $H(\bar\alpha)=0$, \ie, a fixed point for \eqref{general:fixed:point}. As a consequence \eqref{G:bounded} further implies $\bar\alpha \leq a_0-T^{-1}+(\beta\,T^2)^{-1}$. Finally $a_0-T^{-1}$ does not belong to the closure of the set of fixed-points solutions, because if this was the case then the continuity of $G_\beta^{g,\tilde{g}}(\cdot,2)$ would have implied $H(a_0-T^{-1})=0$, clearly in contrast with \eqref{detail}.
\end{enumerate}
\end{proof}

 As a corollary of the previous lemma we have already proven the following
 \begin{corollary}
        If \Cref{A:kappatilde}-$(i)$ holds true, then there exists at least one solution $\alpha^\star_\nu$ on $(\alpha_\nu-T^{-1},\alpha_\nu-T^{-1}+(\beta_\mu\,T^2)^{-1}]$ to the fixed point associated to \eqref{def:sistem:alpha:psi}. Moreover, if $\beta_\mu$ is finite then $\alpha_\nu-T^{-1}$ is not an accumulation point for the set of solutions. 

      Similarly if \Cref{A:kappatilde}-$(ii)$ holds true, there exists at least one solution $\alpha^\star_\mu$ on $(\alpha_\mu-T^{-1},\alpha_\mu-T^{-1}+(\beta_\nu\,T^2)^{-1}]$ to the fixed point associated to \eqref{def:sistem:alpha:varphi}. Moreover, if $\beta_\nu$ is finite then $\alpha_\mu-T^{-1}$ is not an accumulation point for the set of solutions. 
 \end{corollary}

The next result is the counterpart to \Cref{appo:lemma31} and it shows that we do also have an integrated propagation \emph{concavity-to-convexity}.

\begin{lemma}\label{appo:lemma33}
      If \Cref{A:kappatilde}-$(i)$ holds, $\alpha_{\nu,n}>-T^{-1}$ and if
    \be\label{upper:ell:varphi:n}
    \ell_{\varphi^{n+1}}(r)\leq -T^{-1}+r^{-2}\,F_{\beta_\mu}^{g_\mu,\tilde{g}_\nu}(\alpha_{\nu,n},r^2)\eqsp,
    \ee
    then $\alpha_{\nu,n+1}>-T^{-1}$ and for any $r>0$
\bes
\kappa_{\psi^{n+1}}(r)\geq \alpha_{\nu, n+1}-r^{-1}\,\tilde{g}_\nu(r)\eqsp.
\ees

  Similarly if \Cref{A:kappatilde}-$(ii)$ holds, $\alpha_{\mu,n}>-T^{-1}$ and  if
   \bes
    \ell_{\psi^{n}}(r)\leq -T^{-1}+r^{-2}\,F_{\beta_\nu}^{g_\nu,\tilde{g}_\mu}(\alpha_{\mu,n},r^2)\eqsp,
    \ees
    then $\alpha_{\mu,n+1}>-T^{-1}$ and for any $r>0$
    \bes
\kappa_{\varphi^{n+1}}(r)\geq \alpha_{\mu,n+1}-r^{-1}\,\tilde{g}_\mu(r)\eqsp.
\ees
\end{lemma}
\begin{proof}
    We are going to prove only the first part of the lemma since the second can be proven by an analogous reasoning. Firstly, let us consider the function
    \be\label{eq:def:hat:psi:n}
    \hat{\psi}^{n+1}(y)\coloneqq T\psi^{n+1}(y)-TU_\nu(y)+\frac{|y|^2}{2}\eqsp.
    \ee
    Then \eqref{eq:expl:grad} implies that its Hessian is given for any $y\in\bbRD$ by
    \be\label{eq:covariance}
    \nabla^2\hat{\psi}^{n+1}(y)=\frac1T\,\mathrm{Cov}(\pi_T^{y,\varphi^{n+1}})\eqsp,
    \ee
    where we recall from the definition of $\pi^{y,\varphi^{n+1}}_T$  \eqref{eq:cond_distr} that
     \bes
    \pi_T^{y,\varphi^{n+1}}(\De x)\propto\exp\biggl(-\frac{|x-y|^2}{2T}-\varphi^{n+1}(x)\biggr)\,\De x\eqsp.
    \ees
    Moreover, if we set for notations' sake
    $V^{y,n+1}\coloneqq -\log(\De \pi_T^{y,\varphi^{n+1}}/\De\Leb)$,
    then our assumption implies 
    \be\label{eq:appo:ell:conditional}
    \ell_{ V^{y,n+1}}(r)\leq r^{-2}\,F_{\beta_\mu}^{g_\mu,\tilde{g}_\nu}(\alpha_{\nu,n},r^2)\quad\forall\,r>0\eqsp.
    \ee

In order to prove the desired bound for $\kappa_{\psi^{n+1}}$, we will first establish a lower bound for the Hessian $\nabla^2\hat{\psi}^{n+1}$, \ie, a lower bound for the covariance matrix \eqref{eq:covariance}. In view of that, let us consider for any fixed $y\in\bbRD$, the variance $\mathrm{Var}_{X\sim \pi_T^{y,\varphi^{n+1}}}(X_1)$ where $X_i$ denotes the $i^{th}$ scalar component of the random vector $X\sim \pi_T^{y,\varphi^{n+1}}$. Next, observe that
\be\label{dettaglio}
\mathrm{Var}_{X\sim \pi_T^{y,\varphi^{n+1}}}(X_1)\geq \mathbb{E}_{X\sim \pi_T^{y,\varphi^{n+1}}}[\mathrm{Var}_{X\sim \pi_T^{y,\varphi^{n+1}}}(X_1| X_2,\,\dots,\,X_d)]\eqsp,
\ee
and notice that for any given $z=(z_2,\,\dots,\,z_d)\in\rset^{d-1}$ it holds
\bes
\mathrm{Var}_{X\sim \pi_T^{y,\varphi^{n+1}}}(X_1| X_2=z_d,\,\dots,\,X_d=z_d)=\frac12\,\int_{\rset^2}|x-\hat{x}|^2\eqsp\pi_T^{y,\varphi^{n+1}}(\De x|z)\,\pi_T^{y,\varphi^{n+1}}(\De \hat{x}|z)
\ees
where  $(y,z,A)\mapsto \pi_T^{y,h}(A|z)$ is the Markov kernel on $\bbRD\times\rset^{d-1}\times\mathcal{B}(\rset)$ whose transition density w.r.t.\ Lebesgue measure is proportional to $\exp(-V^{y,n+1}(x,z))$. If we set $V^{y,z}(x)\coloneqq V^{y,n+1}(x,z)$ we then have, uniformly in $z\in\rset^{d-1}$,
\bes\begin{split}
1=&\frac12\int (\partial_x V^{y,z}(x)-\partial_x V^{y,z}(\hat{x}))(x-\hat{x})\eqsp\pi_T^{y,\varphi^{n+1}}(\De x|z)\,\pi_T^{y,\varphi^{n+1}}(\De \hat{x}|z)\\
=&\frac12 \int\langle \nabla V^{y,n+1}(x,z)-\nabla V^{y,n+1}(\hat{x}),\eqsp(x,z)-(\hat{x},z)\rangle\eqsp\pi_T^{y,\varphi^{n+1}}(\De x|z)\,\pi_T^{y,\varphi^{n+1}}(\De \hat{x}|z)\\
\overset{\eqref{eq:appo:ell:conditional}}{\leq}&\, \frac12\int F_{\beta_\mu}^{g_\mu,\tilde{g}_\nu}(\alpha_{\nu,n},|x-\hat{x}|^2)\eqsp\pi_T^{y,\varphi^{n+1}}(\De x|z)\,\pi_T^{y,\varphi^{n+1}}(\De \hat{x}|z)\\
\leq&\, F_{\beta_\mu}^{g_\mu,\tilde{g}_\nu}\biggl(\alpha_{\nu,n},2\,\mathrm{Var}_{X\sim \pi_T^{y,\varphi^{n+1}}}(X_1| X_2=z_d,\,\dots,\,X_d=z_d)\biggr)\eqsp,
\end{split}\ees
where the last step follows from the concavity of 
$s\mapsto F_{\beta_\mu}^{g_\mu,\tilde{g}_\nu}(\alpha_{\nu,n},s)$ (cf. \Cref{appo:lemma2}) and Jensen's inequality. By combining the above with \eqref{dettaglio}, since  $\alpha_{\nu,n}>-T^{-1}$ and $F_{\beta_\mu}^{g_\mu,\tilde{g}_\nu}(\alpha_{\nu,n},\cdot)$ is increasing (cf. \Cref{appo:lemma2}), we deduce by definition of $G_{\beta_mu}^{g_\mu,\tilde{g}_\nu}$ that
 \bes
\mathrm{Var}_{X\sim \pi_T^{y,\varphi^{n+1}}}(X_1)\geq \frac12\,G_{\beta_\mu}^{g_\mu,\tilde{g}_\nu}(\alpha_{\nu,n},2)\eqsp.
 \ees

Since the definition $\ell_U$ is invariant under orthonormal transformation, for any orthonormal matrix $O$ the functions $\varphi^{n+1}(O\,\cdot)$ satisfy the condition \eqref{upper:ell:varphi:n} too. The previous bound and this observation leads to 
\bes
\mathrm{Var}_{X\sim \pi_T^{y,\varphi^{n+1}}}(\langle v,\,X\rangle)\geq  \frac12\,G_{\beta_\mu}^{g_\mu,\tilde{g}_\nu}(\alpha_{\nu,n},2)\quad \forall\,y,\,v\in\bbRD\eqsp\text{s.t. }|v|=1\eqsp,
\ees
and hence from \eqref{eq:covariance} we finally deduce
\be\label{eq:lower:hess_psi_hat}
\langle v,\,\nabla^2\hat{\psi}^{n+1}(y)\,v\rangle\geq G_{\beta_\mu}^{g_\mu,\tilde{g}_\nu}(\alpha_{\nu,n},2)\,\frac{|v|^2}{2T}\quad \forall\,y,\,v\in\bbRD\eqsp.
\ee
By recalling \eqref{eq:def:hat:psi:n} and \eqref{def:sistem:alpha:psi}, the above bound concludes our proof.    
\end{proof}

\begin{remark}[A first trivial lower bound]\label{initial:assumption}
Under \Cref{A:kappatilde}, the above discussion already provides a first trivial lower bound for $\kappa_{\psi^n}$. Indeed  \eqref{eq:covariance} tells us that $\hat{\psi}^{n+1}$ is a convex function for any $n\geq 0$, which combined with \eqref{eq:def:hat:psi:n} yields to for $n\geq 0$
\bes
\kappa_{\psi^{n+1}}(r)\geq \alpha_\nu-T^{-1}-r^{-1}\tilde{g}_\nu(r)\eqsp.
\ees
Therefore in \Cref{gio2022:thm} we could consider any initialization $\psi^0$ without prescriptions on its behaviour and run an iteration of Sinkhorn in order to get $\alpha_{\nu,1}\geq \alpha_\nu-T^{-1}$. At this point once can proceed again with the proof of \Cref{gio2022:thm} with the same (but shifted by $-1$) sequence of parameters $(\alpha_{\nu,n})_{n\in\N}$.

Let us notice that the same discussion holds for the sequence of $\varphi^n$, which yields to for $n\geq 0$
\bes
\kappa_{\varphi^{n+1}}\geq \alpha_\mu-T^{-1}-r^{-1}\tilde{g}_\mu(r)\eqsp,
\ees
which for $n=0$ gives for free the base step of the iteration \eqref{def:sistem:alpha:varphi}.

Finally, let us remark here that since the potentials couple $(\varphistar,\psistar)$ can be thought as a constant sequence of Sinkhorn iterates, the above discussion proves also that
\bes
\kappa_{\psistar}(r)\geq \alpha_\nu-T^{-1}-r^{-1}\tilde{g}_\nu(r)\quad\text{ and }\quad\kappa_{\varphistar}\geq \alpha_\mu-T^{-1}-r^{-1}\tilde{g}_\mu(r)\eqsp.
\ees
   \end{remark}

Before proving the main result of this section, let us comment on the differences between
\Cref{appo:lemma31} and \Cref{appo:lemma33}. Both results show how applying the heat semigroup improves (weak) convexity and concavity properties. Actually, since the function $u_s\coloneqq -\log P_{T-s}\exp(-h)$ solves the HJB equation
\bes\begin{cases}
\partial_s u_s+\frac12\Delta_s u_s-\frac12|\nabla u_s|^2=0\\
u_T=h\,,
\end{cases}
\ees
both can be seen as backpropagation results for (weak) convexity and concavity along HJB. Particularly, the core of \Cref{appo:lemma31} is proving that the (weak) convexity of $\psi^n$  back-propagates to $-\log P_{T}\exp(-\psi^n)$, from which we then deduce the (weak) concavity of $\varphi^{n+1}$.  In the strongly log-concave setting this result was already known (see for instance the proof of \cite[Theorem 4]{chewi2022entropic}) and it can be easily deduced from Brascamp-Lieb inequality (see e.g. \cite[Lemma 2]{chewi2022entropic}). Indeed if $\nabla^2\psi^n\geq \alpha_{\nu,n}$ then 
\bes
\nabla^2\biggl(-\log \frac{\De\pi_T^{x,\psi^n}}{\De\Leb}\biggr)= T^{-1}+\nabla^2\psi^n\geq T^{-1}+\alpha_{\nu,n}\eqsp,
\ees
which combined with \eqref{eq:expl:grad} and Brascamp-Lieb inequality yields to
\bes
-\nabla^2 \log P_T\rme^{-\psi^n}(x) = T^{-1}\operatorname{Id}- T^{-2}\,\mathrm{Cov}(\pi_T^{x,\psi^n})\geq T^{-1} - \frac{T^{-1}}{1+T\,\alpha_{\nu,n}}=\frac{\alpha_{\nu,n}}{1+T\,\alpha_{\nu,n}}\eqsp.
\ees
This is exactly what our computations in the proof of \Cref{appo:lemma31} yields to in the strongly log-concave case. Therefore \Cref{appo:lemma31} (combined with \Cref{thm:HJB}) can be considered as a generalisation of the Brascamp-Lieb argument to the weak setting, via a coupling by reflection approach (which stems at the core of the proof of \Cref{thm:HJB}, see \cite[Theorem 2.1]{conforti2022quadratic}).

Similarly, the core of \Cref{appo:lemma33} is proving that the (weak) concavity of $\varphi^{n+1}$  back-propagates to $-\log P_{T}\exp(-\varphi^{n+1})$, from which we then deduce the (weak) convexity of $\psi^{n+1}$. In the strongly log-convex setting, this reduces to applying the Cram\'er-Rao inequality (see Lemma 2 and the proof of Theorem 4 in~\cite{chewi2022entropic}). Indeed in this setting the assumption of \Cref{appo:lemma33} can be rewritten as $\nabla^2\varphi^{n+1}\leq \beta_{n+1}-T^{-1}$ for some $\beta_{n+1}>0$ , which implies 
\bes
\nabla^2\biggl(-\log \frac{\De\pi_T^{y,\varphi^{n+1}}}{\De\Leb}\biggr)= T^{-1}+\nabla^2\varphi^{n+1}\leq \beta_{n+1}\eqsp.
\ees
The latter, combined with \eqref{eq:expl:grad} and the Cram\'er-Rao inequality yields to
\bes
-\nabla^2\log P_T\rme^{-\varphi^{n+1}}(y)=T^{-1}- T^{-2}\,\mathrm{Cov}(\pi_T^{x,\psi^n})\leq T^{-1}-T^{-2}\beta_{n+1}^{-1}\eqsp.
\ees
This last bound implies $\nabla^2\hat{\psi}^{n+1}\geq (T\beta_{n+1})^{-1}$ which coincides with what we have shown in~\eqref{eq:lower:hess_psi_hat}, in this particular setting. Therefore  \Cref{appo:lemma33} can be seen as a generalisation of Cram\'er-Rao inequality to the weak setting.

\medskip 

By consecutively applying \Cref{appo:lemma31} and \Cref{appo:lemma33} we are finally able to prove how  lower-bounds for  integrated convexity profiles propagate along Sinkhorn.

\begin{proof}[Proof of \Cref{gio2022:thm}]
   Let us start showing the first statement. Consider the sequence $(\alpha_{\nu,n})_{n\in\N}$ defined in \eqref{def:sistem:alpha:psi}. We will prove our statement by induction. The case $n=0$ is met under the assumption $\kappa_{\psi^0}(r)\geq\alpha_\nu-T^{1}-r^{-1}\tilde{g}_\nu(r)$. 
The inductive step follows by applying consecutively \Cref{appo:lemma33} and \Cref{appo:lemma31}. As a direct consequence of item $(ii)$ in \Cref{appo:lemma2} we deduce that the sequence $(\alpha_{\nu,n})_{n\in\N}$ is non-decreasing and hence $\alpha_{\nu,n}\geq \alpha_{\nu,0}=\alpha_\nu-T^{-1}$. Since $G_{\beta_\mu}^{g_\mu,\tilde{g}_\nu}$ is continuous and $\alpha_\nu-T^{-1}$ is not an accumulation point for the set of solutions of \eqref{general:fixed:point} (cf. item $(iii)$ in \Cref{appo:lemma2}), we deduce that $\alpha_{\nu,n}> \alpha_{\nu,0}=\alpha_\nu-T^{-1}$ for $n\geq 1$ and that the same holds for its limit $\alpha_{\psistar}$. The upper bound on $\alpha_{\nu,n}$ comes for free from \eqref{def:sistem:alpha:psi} and the upper bound \eqref{G:bounded}. 
The proof of \eqref{eq_bound_gio_teo:psi:limit} is obtained in the same way by considering the (constant) Sinkhorn iterates $(\varphistar,\psistar)$ with the same sequence of $(\alpha_{\nu,n})_{n\in\N}$.

 The proof of the second statement  is completely analogous and for this reason we omit it. The only difference here relies in proving that the base case $n=1$ holds true, but this has been already proven in the discussion of \Cref{initial:assumption}.
\end{proof}

\section{Proof of the main results}

As we have already done in \Cref{sec:main:results}, we will split the discussion in two cases: firstly we address the strongly log-concave case and then we generalise the results to the weakly log-concave setting.

\subsection{Exponential convergence for strongly log-concave marginals}\label{sec:proof:strong}

 Let us firstly observe that \Cref{gio2022:thm} in this particular setting simply reads as 

\begin{theorem}\label{gio2022:thm:concave}
Assume \Cref{A:log:concave:doppio}. Then there exist two monotone increasing sequences $(\alpha_{\mu,n})_{n\in\N}\subseteq(\alpha_\mu-T^{-1},\alpha_\mu-T^{-1}+(\beta_\nu\,T^2)^{-1}]$ and $(\alpha_{\nu,n})_{n\in\N}\subseteq(\alpha_\nu-T^{-1},\alpha_\nu-T^{-1}+(\beta_\mu\,T^2)^{-1}]$ such that for any $n\geq 1$  for any $n\in\N$ it holds $r>0$ it holds
\bes
\kappa_{\varphi^n}(r)\geq \alpha_{\mu,n}\quad\text{ and }\quad\kappa_{\psi^n}(r)\geq \alpha_{\nu,n}\eqsp.
\ees
These two sequences are defined as
\bes
\begin{cases}
\alpha_{\mu,0}\coloneqq\alpha_\mu-T^{-1}\eqsp,\\
\alpha_{\mu,n+1}\coloneqq \alpha_\mu-T^{-1}+\biggl(T^2\,\beta_\nu+(\alpha_{\mu,n}+T^{-1})^{-1}\biggr)^{-1}\eqsp,\quad n\in\N\eqsp,
\end{cases}
\ees
and
\bes
\begin{cases}
\alpha_{\nu,0}\coloneqq\alpha_\nu-T^{-1}\eqsp,\\
\alpha_{\nu,n+1}\coloneqq \alpha_\nu-T^{-1}+\biggl(T^2\,\beta_\mu+(\alpha_{\nu,n}+T^{-1})^{-1}\biggr)^{-1}\eqsp,\quad n\in\N\eqsp.
\end{cases}
\ees
Moreover, both sequences converge respectively to
\be\label{iden:limit:concave}\begin{aligned}
\alpha_{\varphistar}\coloneqq &\,\frac12\biggl(\alpha_\mu+\sqrt{\alpha_\mu^2+4\alpha_\mu/(T^2\beta_\nu)}\biggr)-T^{-1}
\eqsp,\\
\alpha_{\psistar}\coloneqq&\, \frac12\biggl(\alpha_\nu+\sqrt{\alpha_\nu^2+4\alpha_\nu/(T^2\beta_\mu)}\biggr)-T^{-1}\eqsp,
\end{aligned}
\ee
and for any $r >0$ it holds 
\bes
\kappa_{\varphistar}(r)\geq \alpha_{\varphistar}\quad\text{ and }\quad\kappa_{\psistar}(r) > \alpha_{\psistar}\eqsp,
\ees
where $\varphistar$ and $\psistar$ are the Schr\"odinger potentials introduced in \eqref{eq:sch_potentials}.
\end{theorem}
\begin{proof}
This is a particular instance of \Cref{gio2022:thm} when ${\tilde g}_\mu={\tilde g}_\nu={g}_\mu={g}_\nu\equiv 0$. The only statement that does not follow from that theorem is the identification of the limit values $\alpha_{\varphistar}$ and $\alpha_{\psistar}$ in \eqref{iden:limit:concave}. We will only prove the first one since the second identity can be proven in the same way. From \Cref{gio2022:thm} we already know that $\alpha_{\mu,n}\uparrow \alpha_{\varphistar}\in(\alpha_\mu-T^{-1},\alpha_\mu-T^{-1}+(\beta_\nu\,T^2)^{-1}]$. Consider the shifted sequence $\theta^\mu_n\coloneqq \alpha_{\mu,n}+T^{-1}$. Clearly $\theta^\mu_n>0$, $\theta^\mu_n\uparrow \theta^\mu_\infty\coloneqq \alpha_{\varphistar}+T^{-1}$ and the latter limit value can be seen as a fixed point for the iteration
\bes
\theta^\mu_{n+1}= \alpha_\nu+(T^2\,\beta_\nu+(\theta^\mu_n)^{-1})^{-1}\eqsp.
\ees
A straightforward computation shows that the there are just two possible fixed point solutions, namely
\bes
\frac12\biggl(\alpha_\mu-\sqrt{\alpha_\mu^2+4\alpha_\mu/(T^2\beta_\nu)}\biggr)
\eqsp\text{ and }\eqsp\frac12\biggl(\alpha_\mu+\sqrt{\alpha_\mu^2+4\alpha_\mu/(T^2\beta_\nu)}\biggr)\eqsp.
\ees
Since one solution is negative, whereas $(\theta_n^\mu)_{n\in\N}$ is a positive increasing sequence, we immediately deduce that $\theta_\infty^\mu$ equals the largest (and positive) fixed point. This proves \eqref{iden:limit:concave}.
\end{proof}

From the previous result we immediately deduce the
explicit expressions for the rates of convergence appearing in \eqref{it2.1} and \eqref{it2.2} in our proof strategy.

\begin{lemma}\label{lemma:contrazione:strong}
Assume \Cref{A:log:concave:doppio}. Then, for any $x \in\rset^d$, it holds
\be\label{eq:strong:contrazione}\begin{aligned}
\bfW^2_2(\pi_T^{x,\psi^n},\pi_T^{x,\psistar})\leq&\, (\gamma_n^\nu)^2 \int \abs{\nabla \psi^{n} - \nabla \psistar}^2 \rmd \pi_T^{x,\psistar}\,,\\
 \bfW_2^2(\pi_T^{x,\varphi^{n+1}},\pi_T^{x,\varphistar})\leq &\,(\gamma_{n+1}^\mu)^2 \int \abs{\nabla \varphi^{n+1} - \nabla \varphistar}^2 \rmd \pi_T^{x,\varphistar}\,,
\end{aligned}
\ee
with
\bes
\gamma_n^\nu=(\theta_n^\nu)^{-1}=(\alpha_{\nu,n}+T^{-1})^{-1}\quad\text{ and }\quad
\gamma_n^\mu=(\theta_n^\mu)^{-1}=(\alpha_{\mu,n}+T^{-1})^{-1}.
\ees
\end{lemma}
\begin{proof}
 We will prove only the first contraction, since the second can be proven in the same way. \Cref{gio2022:thm:concave} implies the strict  convexity of the (negative) log-densities of $\pi_T^{x,\psi^n}$ and $\pi_T^{x,\psistar}$. Then by recalling that these two probabilities can be seen as the invariant measures for the SDEs 
 \be\label{sistema:invarianti}\begin{cases}
\De X_t=-\biggl(\frac{X_t-x}{2T}+\frac12\nabla \psi^n(X_t)\biggr)\De t+\De B_t\quad \text{with }X_0\sim\pi_T^{x,\psi^n}\,,\\
\De Y_t=-\biggl(\frac{Y_t-x}{2T}+\frac12\nabla \psistar(Y_t)\biggr)\De t+\De B_t\quad \text{with }Y_0\sim\pi_T^{x,\psistar}\,,
 \end{cases}\ee
 the contraction estimate simply follows from a straightforward application of synchronous coupling. 

Let us mention here that \eqref{eq:strong:contrazione} can also be proven by combining Talagrand transport inequality \cite[Corollary 9.3.2]{bakry2013analysis} with Log-Sobolev inequality \cite[Corollary 5.7.2]{bakry2013analysis}.
\end{proof}

Next, notice that from \Cref{gio2022:thm:concave} we can further deduce that the contraction rates obtained above clearly satisfy
\bes
 \begin{cases}
\gamma_0^\mu\coloneqq\alpha_\mu^{-1}\\
\gamma^\mu_{k+1}\coloneqq (\alpha_\mu+(T^2\,\beta_\nu+\gamma^\mu_k)^{-1})^{-1} 
\end{cases}\eqsp\text{ and }\quad \begin{cases}
\gamma_0^\nu\coloneqq\alpha_\nu^{-1}\\
\gamma^\nu_{k+1}\coloneqq (\alpha_\nu+(T^2\,\beta_\mu+\gamma^\nu_k)^{-1})^{-1} \eqsp,
\end{cases}
\ees
and hence are  increasing sequences which converge to the limits 
\bes
\gamma^\mu_\infty\coloneqq 2\biggl(\alpha_\mu+\sqrt{\alpha_\mu^2+4\alpha_\mu/(T^2\beta_\nu)}\biggr)^{-1}
\quad\text{ and }\quad\gamma^\nu_\infty\coloneqq 2\biggl(\alpha_\nu+\sqrt{\alpha_\nu^2+4\alpha_\nu/(T^2\beta_\mu)}\biggr)^{-1}\eqsp.
\ees

Our first main result will then follow by concatenating the previous contraction estimates, as we explain in the following proof.
\begin{theorem}\label{thm:strong:conv}
     Assume \Cref{A:log:concave:doppio}. Then for any $n\in\N$ it holds 
      \be\label{eq:integrated:thm:strong}
\begin{aligned}
\|\nabla \varphi^{n} - \nabla\varphistar\|_{\rmL^2(\mu)} \leq &\eqsp\frac{T}{\gamma_{n}^\mu}\eqsp\prod_{k=0}^{n-1}\frac{\gamma_{k+1}^\mu\eqsp\gamma_k^\nu}{T^{2}}\eqsp \|\nabla \psi^0 - \nabla\psistar\|_{\rmL^2(\nu)}  \eqsp,\\
\|\nabla \psi^{n} - \nabla\psistar\|_{\rmL^2(\nu)} \leq&\eqsp\prod_{k=0}^{n-1}\frac{\gamma_{k+1}^\mu\eqsp\gamma_k^\nu}{T^{2}}\eqsp \|\nabla \psi^0 - \nabla\psistar\|_{\rmL^2(\nu)}   \eqsp,
\end{aligned}
\ee
and 
\be\label{eq:W2:thm}
\begin{aligned}
\bfW_2(\pi^{n,n},\,\pi^\star)+\bfW_2(\pi^{n+1,n},\,\pi^\star)\leq &\eqsp (T+\gamma_n^\nu)\eqsp\prod_{k=0}^{n-1}\frac{\gamma_k^\mu\eqsp\gamma_k^\nu}{T^{2}}\eqsp \ \|\nabla \psi^0 - \nabla\psistar\|_{\rmL^2(\nu)}   \eqsp.
\end{aligned}
\ee
As a consequence, the following entropic bounds hold
\be\label{eq:thm:entropy}
\begin{aligned}
\scrH(\pi^\star|\pi^{n,n})+\scrH(\pi^\star|\pi^{n+1,n})\leq &\biggl(\frac{T^2}{2\,\gamma_{n}^\mu}+\frac{\gamma^\nu_n}{2}\biggr)\eqsp\prod_{k=0}^{n-1}\frac{(\gamma_{k+1}^\mu\eqsp\gamma_k^\nu)^2}{T^{4}}\eqsp \|\nabla \psi^0 - \nabla\psistar\|_{\rmL^2(\nu)}^2 \eqsp.
\end{aligned}
\ee
Finally, for any 
    \be\label{T0:log:concave:again}
T>\frac{\beta_\mu\beta_\nu-\alpha_\mu\alpha_\nu}{\sqrt{\alpha_\mu\,\beta_\mu\,\alpha_\nu\,\beta_\nu\,(\alpha_\mu+\beta_\mu)(\alpha_\nu+\beta_\nu)}}
 \ee
    Sinkhorn's algorithm converges exponentially fast (in $\rmL^2,\,\bfW_2$ and $\scrH$ as in \eqref{eq:W2:exp:conv:general}) with exponential rate of convergence $\lambda\in(\nicefrac{\gamma_{\infty}^\mu\gamma_{\infty}^\nu}{T^2}, 1)$.
\end{theorem} 
\begin{proof}
 Recall that from~\eqref{eq:expl:grad} (which is proven in \Cref{prop:cov} in \Cref{sec:proof-crefprop:cov} under \Cref{A:entropy}, henceforth it is valid also under the stronger \Cref{A:log:concave:doppio}) we have
\bes
\begin{aligned}
\|\nabla \varphi^{n+1} - \nabla\varphistar\|^2_{\rmL^2(\mu)}=\|\nabla \log P_T e^{-\psi^n} - \nabla\log P_T e^{-\psistar}\|^2_{\rmL^2(\mu)}\\
\leq T^{-2}\,\int\bfW_2^2(\pi_T^{x,\psi^n},\pi_T^{x,\psistar})\eqsp \mu(\De x)\eqsp.
\end{aligned}\ees
Combining this bound with the contraction estimate \eqref{eq:strong:contrazione} gives 
\bes
\begin{aligned}
\|\nabla \varphi^{n+1} - \nabla\varphistar\|^2_{\rmL^2(\mu)}\leq \biggl(\frac{\gamma^{\nu}_{n}}{T}\biggr)^2\int |\nabla\psi^n-\nabla\psistar|^2(y)\pi_T^{x,\psistar}(\De y)\mu(\De x) \\
=\biggl(\frac{\gamma^{\nu}_{n}}{T}\biggr)^2\|\nabla\psi^n-\nabla\psistar\|^2_{\rmL^2(\nu)}\eqsp,
\end{aligned}\ees
where the last step holds since $\mu(\De x)\otimes \pi^{x,\psistar}_T(\De y)=\pistarT(\De x\De y)\in\Pi(\mu,\nu)$. Therefore we have shown that 
\bes
\|\nabla \varphi^{n+1} - \nabla\varphistar\|_{\rmL^2(\mu)}\leq 
\frac{\gamma^{\nu}_{n}}{T}\,\|\nabla\psi^n-\nabla\psistar\|_{\rmL^2(\nu)}\eqsp.
\ees
Repeating the same argument but exchanging the roles of $\psi^n,\psistar$ and $\varphi^n,\varphistar$ we obtain
\bes
\|\nabla \psi^{n} - \nabla\psistar\|_{\rmL^2(\nu)}\leq \frac{\gamma^{\mu}_{n-1}}{T}\,\| \nabla\varphi^n-\nabla\varphistar\|_{\rmL^2(\mu)}\eqsp.
\ees
Combining these last two bounds yields to \bes
\|\nabla \varphi^{n+1} - \nabla\varphistar\|_{\rmL^2(\mu)}
\leq \frac{\gamma^{\nu}_{n}\gamma^{\mu}_{n-1}}{T^2} \|\nabla \varphi^{n} - \nabla\varphistar\|_{\rmL^2(\mu)}\eqsp,
\ees
which iterated over $n$ gives the first estimate of~\eqref{eq:integrated:thm:strong}. The second one can be proven in the same way.
 The Wasserstein bound \eqref{eq:W2:thm} follows from \Cref{lemma:contrazione:strong} and \eqref{eq:integrated:thm:strong}.
    
    The entropic bound can be deduced as follows.
     Since $\pi^{n+1,n}\in\Pi(\mu,\star)$, the additive property of the relative entropy yields to
    \bes
    \scrH(\pistarT|\pi^{n+1,n})=\int \scrH( \pi_T^{x,\psistar}|\pi_T^{x,\psi^n})\,\mu(\De x)\,,
    \ees
    where $\pi_T^{x,\psi^n}$ and $\pi_T^{x,\psistar}$ are the conditional probability measures, \ie, such that  $\pi^{n+1,n}(\De x\De y)= \mu(\De x)\pi_T^{x,\psi^n}(\De y)$ and $\pistarT(\De x\De y)=\mu(\De x) \pi_T^{x,\psistar}(\De y)$ and whose  densities read as  
 \bes
    \pi^{x,h}_T(\De y)\propto \exp\biggl(-\frac{|y-x|^2}{2T}-h(y)\biggr) \De y \quad\text{ for }h\in\{\psistar,\,\psi^n\}\eqsp.
    \ees
    Next, notice that \Cref{gio2022:thm:concave} implies the log-concavity $\pi_T^{x,\psi^n}$ since 
    \be\label{eq:hess:log:density}
    \nabla^2(-\log\pi_T^{x,\psi^n}(\De y))=T^{-1}+\nabla^2\psi^n\geq T^{-1}+\alpha_{\nu,n}=(\gamma^\nu_n)^{-1}\,,
    \ee
    and therefore this measure satisfies the Log-Sobolev inequality with constant $\gamma^\nu_n$ \cite[Corollary 5.7.2]{bakry2013analysis}. This implies that
    \bes
    \begin{split}
\scrH(\pi_T^{x,\psistar}|\pi_T^{x,\psi^n})\leq \frac{\gamma^\nu_n}{2}\, \biggl\|\nabla\log\frac{\De \pi_T^{X,\psistar}}{\De \pi_T^{X,\psi^n}}\biggr\|^2_{\rmL^2(\pi_T^{x,\psistar})}
    =\frac{\gamma^\nu_n}{2}\,\|\nabla\psi^n-\nabla\psistar\|^2_{\rmL^2(\pi_T^{x,\psistar})}\,,
    \end{split}
    \ees
and therefore we have shown that 
\bes
\begin{aligned}
\scrH(\pistarT|\pi^{n+1,n})\leq &\,\frac{\gamma^\nu_n}{2}\,\int\|\nabla\psi^n-\nabla\psistar\|^2_{\rmL^2(\pi_T^{X,\psistar})}\,\mu(\De x)\\
=&\,\frac{\gamma^\nu_n}{2}\,\int \int |\nabla\psi^n-\nabla\psistar|^2(y)\,\De \pi_T^{x,\psistar}(\De y)\mu(\De x)\\
=&\,\frac{\gamma^\nu_n}{2}\,\int |\nabla\psi^n-\nabla\psistar|^2(y)\pistarT(\De x\De y)\\
=&\,\frac{\gamma^\nu_n}{2}\,\int |\nabla\psi^n-\nabla\psistar|^2\De\nu\\
=&\,\frac{\gamma^\nu_n}{2}\,\|\nabla\psi^n-\nabla\psistar\|^2_{\rmL^2(\nu)}\,.
\end{aligned}
\ees
 This proves 
 \bes
\scrH(\pi^\star|\pi^{n,n})\leq \frac{T^2}{2\,\gamma_{n}^\mu}\eqsp\prod_{k=0}^{n-1}\frac{(\gamma_{k+1}^\mu\eqsp\gamma_k^\nu)^2}{T^{4}}\eqsp \|\nabla \psi^0 - \nabla\psistar\|_{\rmL^2(\nu)}^2 \eqsp.
\ees
By reasoning in the same fashion we can prove that 
\bes
\scrH(\pi^\star|\pi^{n+1,n})\leq\frac{\gamma^\nu_n}{2}\eqsp\prod_{k=0}^{n-1}\frac{(\gamma_{k+1}^\mu\eqsp\gamma_k^\nu)^2}{T^{4}}\eqsp \|\nabla \psi^0 - \nabla\psistar\|^2_{\rmL^2(\nu)}   \eqsp,
\ees
which yields to \eqref{eq:thm:entropy}.

   Finally, the exponential convergence condition \eqref{T0:log:concave:again} holds if $ T^{-2}\gamma_{\infty}^\mu\,\gamma_{\infty}^\nu < 1$. This is equivalent to $T>\theta\,\gamma_\infty^\mu$ and $T>\theta^{-1}\,\gamma_\infty^\nu$ for some $\theta>0$,  \ie,  that for any $\theta\in(0,\infty)$ it holds 
\be\label{low:ineq}
\begin{aligned}
T>\theta\,\alpha_\mu^{-1}-\theta^{-1}\beta_\nu^{-1}\quad&\Leftrightarrow\quad T>\theta \,\gamma_\infty^\mu\eqsp,\\
T>\theta^{-1}\alpha_\nu^{-1}-\theta\,\beta_\mu^{-1}\quad&\Leftrightarrow\quad T>\theta^{-1}\,\gamma_\infty^\nu\eqsp.
\end{aligned}\ee
Therefore if 
\bes
\begin{split}
T>&\inf_{\theta\in(0,\infty)}\max\{\theta\,\alpha_\mu^{-1}-\theta^{-1}\beta_\nu^{-1},\,\theta^{-1}\alpha_\nu^{-1}-\theta\,\beta_\mu^{-1}\}
=\frac{\alpha_\mu^{-1}\alpha_\nu^{-1}-\beta_\mu^{-1}\beta_\nu^{-1}}{\sqrt{(\alpha_\mu^{-1}+\beta_\mu^{-1})(\alpha_\nu^{-1}+\beta_\nu^{-1})}}\\
=&\frac{\beta_\mu\beta_\nu-\alpha_\mu\alpha_\nu}{\sqrt{\alpha_\mu\,\beta_\mu\,\alpha_\nu\,\beta_\nu\,(\alpha_\mu+\beta_\mu)(\alpha_\nu+\beta_\nu)}}\eqsp,
\end{split}
\ees
then we are guaranteed that $T^2>\gamma_\infty^\mu\,\gamma_\infty^\nu$, whence the exponential convergence of Sinkhorn's algorithm. 
\end{proof}

\subsection{Exponential convergence for weakly log-concave marginals}\label{sec:proof:weak}

Now we generalise the previous results and adapt our proof strategy to the weakly log-concave setting. For exposition's sake we have postponed some results and derivations to \Cref{new:sec}.

Let us start with contraction estimates akin \Cref{lemma:contrazione:strong}.

\begin{lemma}\label{lemma:expl:sticky}
Assume \Cref{A:entropy}. Then, for any $x \in\rset^d$, it holds
\be\label{ineq:W1:psi}
\bfW_1(\pi_T^{x,\psi^n},\pi_T^{x,\psistar})\leq \gamma_n^\nu \int \abs{\nabla \psi^{n} - \nabla \psistar} \rmd \pi_T^{x,\psistar},
\ee
and similarly
\be\label{ineq:W1:phi}
\bfW_1(\pi_T^{x,\varphi^{n+1}},\pi_T^{x,\varphistar})\leq \gamma_{n+1}^\mu \int \abs{\nabla \varphi^{n+1} - \nabla \varphistar} \rmd \pi_T^{x,\varphistar},
\ee
where for any $\frp\in\{\mu,\nu\}$, and for any $n\in\N$ the contraction rate $\gamma^\frp_n$ is defined as
    \be\label{eq:def:rate:1}
    \gamma^\frp_n\coloneqq \frac{\tilde g_\frp{'}(0)^2}{\tilde g_\frp{'}\biggl(\norm{\tilde g_\frp}_\infty\biggl(\frac{1}{\tilde g_\frp{'}(0)}+\frac{2}{\alpha_{\frp,n}+T^{-1}}\biggr)\biggr)^2}\frac{1}{\alpha_{\frp,n}+T^{-1}+\tilde g_\frp{'}(0)}\eqsp
    \ee
    where  $(\alpha_{\mu,n})_{n\in\N}\subseteq (\alpha_\mu-T^{-1},\alpha_\mu-T^{-1}+(\beta_\nu\,T^2)^{-1}]$ and $(\alpha_{\nu,n})_{n\in\N}\subseteq (\alpha_\nu-T^{-1},\alpha_\nu-T^{-1}+(\beta_\mu\,T^2)^{-1}]$ are monotone increasing sequences, built in \Cref{gio2022:thm}.
\end{lemma}

This result can be proven mimicking the ideas employed in \Cref{lemma:contrazione:strong}, this time considering a reflection coupling instead of the synchronous coupling in order to deal with the weak convexity of the energies associated to the invariant SDEs \eqref{sistema:invarianti}. For exposition's sake we defer the proof of the above lemma to \Cref{new:sec}.

\medskip

 Recall that the monotone sequences $(\alpha_{\mu,n})_{n\in\N}\subseteq (\alpha_\mu-T^{-1},\alpha_\mu-T^{-1}+(\beta_\nu\,T^2)^{-1}]$ and $(\alpha_{\nu,n})_{n\in\N}\subseteq (\alpha_\nu-T^{-1},\alpha_\nu-T^{-1}+(\beta_\mu\,T^2)^{-1}]$  built in \Cref{gio2022:thm} converges respectively to $\alpha_\varphistar$ and $\alpha_\psistar$. Therefore the contraction rates sequences $(\gamma^\mu_n)_{n\in\N}$ and $(\gamma^\nu_n)_{n\in\N}$ defined above converge respectively to the asymptotic rates 
 \bes
     \begin{aligned}
    \gamma^\mu_\infty\coloneqq &\eqsp\frac{\tilde g_\mu{'}(0)^2}{\tilde g_\mu{'}\biggl(\norm{\tilde g_\mu}_\infty\biggl(\frac{1}{\tilde g_\mu{'}(0)}+\frac{2}{\alpha_{\varphistar}+T^{-1}}\biggr)\biggr)^2}\frac{1}{\alpha_{\varphistar}+T^{-1}+\tilde g_\mu{'}(0)}\eqsp,\\
    \gamma^\nu_\infty\coloneqq &\eqsp\frac{\tilde g_\nu{'}(0)^2}{\tilde g_\nu{'}\biggl(\norm{\tilde g_\nu}_\infty\biggl(\frac{1}{\tilde g_\nu{'}(0)}+\frac{2}{\alpha_{\psistar}+T^{-1}}\biggr)\biggr)^2}\frac{1}{\alpha_{\psistar}+T^{-1}+\tilde g_\nu{'}(0)}\eqsp.
    \end{aligned}
    \ees

By concatenating the previous contraction estimates, and by combining them with our proof strategy we are finally able to prove our main result in the weakly log-concave setting.

\begin{theorem}\label{thm:integrated:convergence} 
Assume that \Cref{A:entropy} holds. Then for any $n\in\N$ it holds 
     \be\label{eq:integrated:thm}
\begin{aligned}
\|\nabla \varphi^{n} - \nabla\varphistar\|_{\rmL^1(\mu)} \leq &\eqsp\frac{T}{\gamma_{n}^\mu}\eqsp\prod_{k=0}^{n-1}\frac{\gamma_{k+1}^\mu\eqsp\gamma_k^\nu}{T^{2}}\eqsp \|\nabla \psi^0 - \nabla\psistar\|_{\rmL^1(\nu)}  \eqsp,\\
\|\nabla \psi^{n} - \nabla\psistar\|_{\rmL^1(\nu)} \leq&\eqsp\prod_{k=0}^{n-1}\frac{\gamma_{k+1}^\mu\eqsp\gamma_k^\nu}{T^{2}}\eqsp \|\nabla \psi^0 - \nabla\psistar\|_{\rmL^1(\nu)}   \eqsp,
\end{aligned}
\ee
and
\be\label{eq:W1:thm}
\begin{aligned}
\bfW_1(\pi^{n,n},\,\pi^\star)+\bfW_1(\pi^{n+1,n},\,\pi^\star)\leq &\eqsp (T+\gamma_n^\nu)\eqsp\prod_{k=0}^{n-1}\frac{\gamma_{k+1}^\mu\eqsp\gamma_k^\nu}{T^{2}}\eqsp \ \|\nabla \psi^0 - \nabla\psistar\|_{\rmL^1(\nu)}   \eqsp,
\end{aligned}
\ee
 where $(\gamma_k^\mu)_{k\in\N}$ and $(\gamma_k^\nu)_{k\in\N}$ are non-negative non-increasing sequences given in \eqref{eq:def:rate:1}. As a consequence, for any
  \be\label{eq:suff:T}
    T^2>\frac{\tilde g_\mu{'}(0)\,(\alpha_\mu+\tilde g_\mu{'}(0))^{-1}}{\tilde g_\mu{'}\biggl(\norm{\tilde g_\mu}_\infty\biggl(\frac{1}{\tilde g_\mu{'}(0)}+\frac{2}{\alpha_\mu}\biggr)\biggr)^2}\eqsp \frac{\tilde g_\nu{'}(0)\,(\alpha_\nu+\tilde g_\nu{'}(0))^{-1}}{\tilde g_\nu{'}\biggl(\norm{\tilde g_\nu}_\infty\biggl(\frac{1}{\tilde g_\nu{'}(0)}+\frac{2}{\alpha_\nu}\biggr)\biggr)^2}
    \eqsp.
    \ee
   Sinkhorn's algorithm converges exponentially fast (in $\rmL^1$ and $\bfW_1$ as in~\eqref{informal:convergence}) with exponential rate of convergence  $\lambda\in(\nicefrac{\gamma_{\infty}^\mu\gamma_{\infty}^\nu}{T^2},1)$.
\end{theorem}
\begin{proof}
    The proof of this result follows the same line portrayed in the proof of \Cref{thm:strong:conv}. The condition \eqref{eq:suff:T} is a sufficient condition for the exponential convergence of Sinkhorn's algorithm which holds
     as soon as $T^2> \gamma_\infty^\mu\,\gamma_\infty^\nu$.
\end{proof}

\begin{appendix}
\section{Wasserstein distance w.r.t\ a measure with log-concave profile}\label{new:sec}
In this section we consider two probability measures $\frp,\,\frq\in\cP(\bbRD)$  that can be again written with log-densities as
\bes
\frp(\De x)=\exp(-U_\frp(x))\De x\,,\qquad\frq(\De x)=\exp(-U_\frq(x))\De x\eqsp.
\ees

Throughout this section we will assume that $\frp\in\mcpalcd$ that we recall here means that $U_\frp$ has an integrated convex profile, \ie, that there exist some $\alpha_\frp>0$ and ${\tilde g}_\frp\in\tmcg$ such that
        \bes
\kappa_{U_\frp}(r)\geq \alpha_\frp-r^{-1}\,{\tilde g}_\frp(r)\quad \forall r>0\eqsp.
\ees
As far as regards $\frq\in\cP(\bbRD)$ we will solely assume $U_\frq\in\cC^1(\bbRD)$ to be a coercive, \ie, that there exist $\gamma_\frq>0$ and $R_\frq\geq 0$ such that
        \be\label{eq:def.Uq:coercive}
        -\frac12\langle \nabla U_\frq(x),x\rangle\leq -\gamma_\frq\,|x|^2\quad\forall\, |x|\geq R_\frq\eqsp. 
        \ee

Let us also emphasize here  that the convexity of integrated profile assumption is stronger than the coercivity, since the former implies
\bes
-\frac12\langle \nabla U_\frp(x),x\rangle\leq -\frac{\alpha_\frp}{2}\,|x|^2+\frac{|\nabla U_\frp(0)|+\| {\tilde g}_\frp\|_{\infty}}{2}|x|\eqsp,
\ees
and hence that the coercive condition holds
\bes
-\frac12\langle \nabla U_\frp(x),x\rangle\leq -\gamma_\frp |x|^2\quad\forall\, |x|\geq R_\frp
\ees
for some $\gamma_\frp>0$ and $R_\frp>0$. 
Notice that $\frp$ and $\frq$ can be seen as invariant measures of the corresponding SDEs
\be\label{SDE:p}
\begin{cases}\De X_t=-\frac12\nabla U_\frp(X_t)\,\De t+\De B_t\eqsp,\\
\De Y_t=-\frac12\nabla U_\frq(Y_t)\,\De t+\De B_t\eqsp,
\end{cases}\ee
which admit unique strong solutions, in view of the coercivity of the corresponding drifts and owing to  \cite[Theorem 2.1]{tweedie96}. Finally let $(P_t^\frp)_{t\geq 0}$ and $(P_t^\frq)_{t\geq 0}$ denote the corresponding Markov semigroups associated to the above SDEs. Since $\frp$ and $\frq$ are invariant measures we clearly have $\frp P^\frp_t=\frp$ and $\frq P_t^\frq=\frq$ for any $t\geq 0$.

\medskip

The main result of this section is showing how the Wasserstein distance between $\frp$ and $\frq$ can be bounded w.r.t. the integrated difference of the drifts appearing in \eqref{SDE:p}.

\begin{theorem}\label{teo:general}
Assume $\frp\in\mcpalcd$ and $U_\frq$ to be coercive (as in~\eqref{eq:def.Uq:coercive}). Then it holds
    \bes
    \begin{split}
    \bfW_1(\frp,\eqsp\frq)\leq& \biggl(\frac{\tilde g_\frp{'}(0)}{\tilde g_\frp{'}(R)}\biggr)^2\frac{1}{\alpha_\frp+\tilde g_\frp{'}(0)}\eqsp\int|\nabla U_\frp-\nabla U_\frq|\De\frq\\
    =&\biggl(\frac{\tilde g_\frp{'}(0)}{\tilde g_\frp{'}(R)}\biggr)^2\frac{1}{\alpha_\frp+\tilde g_\frp{'}(0)}\eqsp\int\left|\nabla \log\frac{\De\frp}{\De\frq}\right|\De\frq\eqsp,
    \end{split}
    \ees
    with $R\coloneqq \norm{\tilde g_\frp}_\infty((\tilde g_\frp{'}(0))^{-1}+2/\alpha_\frp)$. 
\end{theorem}
\begin{proof}
Firstly,  let us consider the function
\bes
f(r)\coloneqq \begin{cases}
    \tilde g_\frp(r)\quad&\text{ if }r\leq R\eqsp,\\
    \tilde g_\frp(R)+\tilde g_\frp{'}(R)\,(r-R)\quad&\text{ otherwise.}
\end{cases}
\ees
Notice that $f(0)=0$ and that $f\in\cC^1((0,+\infty),\rset_+)\cap\cC^2((0,R)\cup(R,+\infty),\rset_+)$ with $f{''}$ having a jump discontinuity in $r=R$. Moreover, $f$ is non-decreasing and concave and equivalent to the identity \ie, for any $r>0$ it holds
\be\label{eq:equiv:f}
\tilde g_\frp{'}(R)\,r\leq f(r)\leq \tilde g_\frp{'}(0)\,r\eqsp,\quad\text{ since }\eqsp \tilde g_\frp{'}(R)\leq f{'}(r)\leq \tilde g_\frp{'}(0)\eqsp.
\ee

Notice that, since $\tilde g_\frp\in\tilde \cG$, for any $r\in(0,R)$ it holds
\bes
\begin{aligned}
2\eqsp f^{''}(r)-\frac{r_t\eqsp f^{'}(r)}{2}\kappa_{U_\frp}(r)\leq&\, 2\eqsp \tilde g_\frp^{''}(r) -\frac{\alpha_\frp}{2}\,\tilde g_\frp{'}(r)\,r+\frac12{\tilde g}_\frp(r)\tilde g_\frp{'}(r)\leq  -\frac{\tilde g_\frp{'}(r)}{2}\,(\alpha_\frp\,r+\tilde g_\frp(r))\\
\leq &\,-\frac{\tilde g_\frp{'}(R)}{2}\,\biggl(\frac{\alpha_\frp}{\tilde g_\frp{'}(0)}+1\biggr)\eqsp\tilde g_\frp(r)=-\frac{\tilde g_\frp{'}(R)}{2}\,\biggl(\frac{\alpha_\frp}{\tilde g_\frp{'}(0)}+1\biggr)\eqsp f(r)\eqsp;
\end{aligned}
\ees
whereas for any $r>R$ it holds 
\bes
\begin{aligned}
2\eqsp f^{''}(r)-\frac{r_t\eqsp f^{'}(r)}{2}\kappa_{U_\frp}(r)=&\,-\frac{\tilde g_\frp{'}(R)}{2}\eqsp r\,\kappa_{U_\frp}(r)\leq -\frac{\tilde g_\frp{'}(R)}{2}\,(\alpha_\frp\,r-\tilde g_\frp(r))\\
\leq&\,  -\frac{\tilde g_\frp{'}(R)}{2}\,\alpha_\frp\,\frac{R}{\tilde g_\frp(R)}\,f(r)+\frac{\tilde g_\frp{'}(R)}{2}\,f(r)\eqsp,
\end{aligned}
\ees
where the last step follows from the concavity of $\tilde g_\frp$, which implies for any $r\geq R$ that $\tilde g_\frp(r)\leq f(r)$ and the monotonicity of $F(r)\coloneqq r/f(r)$, since for any $r\geq R$ it holds
\bes
F{'}(r)=\frac{f(r)-r\,f{'}(r)}{f(r)^2}=\frac{\tilde g_\frp(R)-R\,\tilde g_\frp{'}(R)}{f(r)^2}\geq 0\eqsp.
\ees
Therefore, by recalling the definition of $R$, for any $r\in(0,R)\cup(R,+\infty)$ we have shown that 
\be\label{eq:appo:diff:general}
2\eqsp f^{''}(r)-\frac{r_t\eqsp f^{'}(r)}{2}\kappa_{U_\frp}(r)\leq -\lambda_\frp\,f(r)\quad\text{ with }\eqsp
\lambda_\frp\coloneqq \frac{\tilde g_\frp{'}(R)}{2}\,\biggl(\frac{\alpha_\frp}{\tilde g_\frp{'}(0)}+1\biggr)\eqsp.
\ee

By relying on the above construction, we will prove our thesis considering the Wasserstein distance 
\bes
   \bfW_{f}(\frp,\,\frq)\coloneqq \inf_{\pi\in \Pi(\frp,\frq)}\mathbb{E}_{(X,Y)\sim\pi}\bigl[f(|X-Y|)\bigr]
   \ees
 induced by the concave function $f$. The above is indeed a distance since $f(0)=0$, $f$ is strictly increasing, concave and hence also subadditive (which implies the triangular inequality). Moreover, from \eqref{eq:equiv:f} it follows the equivalence between  $\bfW_{f}$ and the usual $\bfW_1$, namely
\be\label{eq:appo:equivalence:wasser:general}
\tilde g_\frp{'}(R)\eqsp\bfW_1(\frp,\,\frq)\leq \bfW_{f}(\frp,\,\frq)\leq \tilde g_\frp{'}(0)\eqsp\bfW_1(\frp,\,\frq)\eqsp,
\ee

For any $t\geq 0$ notice that
\be\label{eq:triangular:wasser:general}
\begin{split}
\bfW_{{f}}(\frp,\,\frq)\leq \bfW_{{f}}(\frp,\,\frq\,P_t^\frp)+\bfW_{{f}}(\frq\,P_t^\frp ,\,\frq)
=\bfW_{{f}}(\frp\,P_t^\frp,\,\frq\,P_t^\frp)+\bfW_{{f}}(\frq\,P_t^\frp ,\,\frq)\eqsp.
\end{split}\ee

In order to bound the first Wasserstein distance appearing in the upper bound \eqref{eq:triangular:wasser:general}, namely $\bfW_{{f}}(\frp\,P_t^\frp,\,\frq\,P_t^\frp)$, fix an initial coupling $(X_0,X^\frq_0)\sim \pi^\star$ distributed according to the optimal coupling for $\bfW_{{f}}(\frp,\,\frq)$
and consider the reflection coupling diffusion processes
\bes
\begin{cases}
    \De X_t=-\frac12\nabla U_\frp(X_t)\,\De t+\De B_t\\
    \De X^\frq_t=-\frac12\nabla U_\frp(X^\frq_t)\,\De t+\De \hat{B}_t\quad\forall\eqsp t\in[0,\tau)\eqsp\text{ and }X_t=X^\frq_t\quad\forall \eqsp t\geq \tau\\
    (X_0,X^\frq_0)\sim \pi^\star\eqsp,
\end{cases}
\ees
where $\tau\coloneqq \inf\{s\geq 0:X^\frq_s=X_s\}$, 
and $(\hat{B}_t)_{t\geq 0}$ is defined as
\begin{equation*}
\De \hat{B}_t\coloneqq (\operatorname{Id}-2\,e_t\,e_t^{{\sf  T}}\,\mathbf{1}_{\{t< \tau\}})\, \De B_t\qquad\text{where}\quad e_t\coloneqq \begin{cases}\frac{Z_t}{|Z_t|}\quad&\text{ when } r_t>0\,,\\
   u\quad&\text{ when }r_t=0\,.\end{cases}
\end{equation*}
where $Z_t\coloneqq X_t-X^\frq_t$, $r_t\coloneqq |Z_t|$ and $u\in\bbRD$ is a fixed (arbitrary) unit-vector. By L\'evy's characterization, $(\hat{B}_t)_{t\geq 0}$ is a $d$-dimensional Brownian motion. As a result, $X_t\sim \frp$ and $X^\frq_t\sim \frq P^\frp_t$ for any $t\geq0$.  In addition $\De W_t\coloneqq e_t^{{\sf T}}\De B_t$ is a one-dimensional Brownian motion. Let us notice that  for any $t<\tau$ it holds
\bes\begin{aligned}
\De Z_t=&\,-2^{-1}\,(\nabla U_\frp(X_t)-\nabla U_\frp(X^\frq_t))\,\De t+2\eqsp e_t\eqsp\De W_t\eqsp,\\
\De r_t^2=&\,-\langle Z_t,\eqsp\nabla U_\frp(X_t)-\nabla U_\frp(X^\frq_t)\rangle\eqsp\De t+4\eqsp\De t+4\eqsp \langle Z_t,\eqsp e_t\rangle\eqsp\De W_t\eqsp,\\
\De r_t=&\,-2^{-1}\,\langle e_t,\eqsp\nabla U_\frp(X_t)-\nabla U_\frp(X^\frq_t)\rangle\eqsp\De t+2\eqsp\De W_t\eqsp.
\end{aligned}
\ees

Now, an application of Ito-Tanaka formula \cite[Chapter VI, Theorem 1.5]{revuzyor}  to the concave function $f\in\cC^1((0,+\infty),\rset_+)\cap\cC^2((0,R)\cup(R,+\infty),\rset_+)$, gives for any $t<\tau$
\be\label{eq:Ito:Tanaka}
f(r_t)=f(r_0)+\int_0^t f{'}(r_s)\,\De r_s+\frac12\int_\rset L_t^a\,\mu_{f}(\De a)\eqsp,
\ee
where $(L_t^a)_{t}$ denotes the right-continuous local time of the semimartingale $(r_t)_t$, whereas $\mu_{f}$ is the non-positive measure representing $f{''}$ in the sense of distributions, \ie, $\mu_{f}([a,b])=f{'}(b)-f{'}(a)$ for any $a\leq b$. Let us further notice that the Meyer Wang occupation times formula \cite[Theorem 29.5]{kallenberg:Modern:Prob}, which for any measurable function $H\colon\rset\to [0,+\infty)$ reads as
\bes
\int_0^t H(r_s)\,\De[r]_s=\int_\rset H(a) L_t^a\,\De a\eqsp,
\ees
implies that the random set $\{s\in[0,\tau]\colon r_s=R\}$ has almost surely zero Lebesgue measure.
Particularly, since $\mu_{f}$ is non-positive, this combined with the above formula implies
\be\label{eq:occopation}
\frac12\int_\rset L_t^a\,\mu_{f}(\De a)\leq \frac12\int_\rset \mathbf{1}_{\{a\neq R\}} 
 L_t^af{''}(a)\De a=2\int_0^t\mathbf{1}_{\{r_s\neq R\}}f{''}(r_s)\De s=2\int_0^t\,f{''}(r_s)\De s.
\ee

As a byproduct of \eqref{eq:Ito:Tanaka} and \eqref{eq:occopation} we may finally state that for any $t<\tau$ it almost surely holds
\bes
\begin{split}
\De f(r_t)\leq&-\frac{f^{'}(r_t)}{2}\, \langle e_t,\eqsp\nabla U_\frp(X_t)-\nabla U_\frp(X^\frq_t)\rangle\eqsp\De t+2\eqsp f^{''}(r_t)\eqsp\De t+2\eqsp f^{'}(r_t)\eqsp\De W_t\\
\leq &\biggl(2\eqsp f^{''}(r_t)-\frac{r_t\eqsp f^{'}(r_t)}{2}\kappa_{U_\frp}(r_t)\biggr)\eqsp\De t+2\eqsp f^{'}(r_t)\eqsp\De W_t\\
\overset{\eqref{eq:appo:diff:general}}{\leq}&\eqsp -\lambda_\frp\eqsp f(r_t)\eqsp\De t+2\eqsp f^{'}(r_t)\eqsp\De W_t\eqsp.
\end{split}
\ees
By recalling that $f(r_t)=f(0)=0$ as soon as $t\geq\tau$, by taking expectation, integrating over time and Gronwall Lemma we have finally proven that for any $t\geq 0$
\be\label{eq:appo:prima:wasser:general}
\begin{split}
\bfW_{{f}}(\frp\,P_t^\frp,\,\frq\,P_t^\frp)\leq  \mathbb{E}[{f}(|X_t-X^\frq_t|)]
\leq e^{-\lambda_\frp \, t}\eqsp \mathbb{E}[f(|X_0-X^\frq_0|)]
= e^{-\lambda_\frp \, t}\eqsp \bfW_{{f}}(\frp,\,\frq)\,.
\end{split}
\ee

\medskip

Let us now provide a bound for the second Wasserstein distance appearing in the upper bound~\eqref{eq:triangular:wasser:general}, namely $\bfW_{{f}}(\frq\,P_t^\frp ,\,\frq)$, by relying on synchronous coupling technique. Therefore fix an initial random variable $Y_0\sim \frq$ and a $d$-dimensional Brownian motion $(B_t)_{t\geq 0}$, and consider now the diffusion processes 
\bes
\begin{cases}
    \De X^\frq_t=-\frac12\nabla U_\frp(X^\frq_t)\,\De t+\De B_t\\
    \De Y_t=-\frac12\nabla U_\frq(Y_t)\,\De t+\De B_t\\
    X^\frq_0=Y_0\sim \frq \eqsp.
\end{cases}
\ees
Notice that for any $t\geq 0$, $Y_t\sim \frq$ whereas $X^\frq_t\sim \frq P^\frp_t$. 
If we set now $\bar r_t\coloneqq |X^\frq_t-Y_t|$ we then have 
\bes
\begin{aligned}
\De (X^\frq_t-Y_t)=&\eqsp-2^{-1}(\nabla U_\frp(X^\frq_t)-\nabla U_\frq(Y_t))\eqsp\De t\eqsp,\\
\De \bar r_t^2=&\eqsp-\langle X^\frq_t-Y_t,\eqsp \nabla U_\frp (X^\frq_t)-\nabla U_\frq(Y_t)\rangle\eqsp\De t\eqsp.
\end{aligned}
\ees
At this point we would like to apply the square-root function, however the latter fails to be $\rmC^2$ in the origin whereas $\bar r_t$ may be equal to zero (e.g., we already start with $\bar r_0=0$). For this reason we are going to perform an approximation argument. Fix $\delta>0$ and consider the function $\rho_\delta(r)\coloneqq \sqrt{r+\delta}$. Then it holds
\bes
\begin{aligned}
    \De \rho_\delta(\bar r_t^2)=&\eqsp -(2\eqsp \rho_\delta(\bar r_t^2))^{-1}\langle X^\frq_t-Y_t,\eqsp \nabla U_\frp (X^\frq_t)-\nabla U_\frq(Y_t)\rangle\eqsp\De t\\
    =&\eqsp -(2\eqsp \rho_\delta(\bar r_t^2))^{-1}\langle X^\frq_t-Y_t,\eqsp \nabla U_\frp (X^\frq_t)-\nabla U_\frp(Y_t)\rangle\eqsp\De t\\
&\qquad\qquad\qquad\qquad-(2\eqsp \rho_\delta(\bar r_t^2))^{-1}\langle X^\frq_t-Y_t,\eqsp \nabla U_\frp (Y_t)-\nabla U_\frq(Y_t)\rangle\eqsp\De t\\
\leq &\eqsp -2^{-1}\eqsp\frac{\bar r_t^2}{\rho_\delta(\bar r_t^2)}\eqsp\kappa_\frp(\bar r_t)\eqsp\De t+2^{-1}\eqsp\frac{\bar r_t}{\rho_\delta(\bar r_t^2)}\eqsp| \nabla U_\frp-\nabla U_\frq|(Y_t)\De t\\
\leq &\eqsp -2^{-1}\eqsp\frac{\bar r_t^2}{\rho_\delta(\bar r_t^2)}\eqsp(\alpha_\frp- \tilde G_\frp)\eqsp\De t+2^{-1}\eqsp\frac{\bar r_t}{\rho_\delta(\bar r_t^2)}\eqsp| \nabla U_\frp-\nabla U_\frq|(Y_t)\De t\eqsp,
\end{aligned}
\ees
where in the last step we have relied on the fact that ${\tilde g}_\frp\in\tmcg$ implies ${\tilde g}_\frp(r)\leq \tilde G_\frp\,r$ for some positive constant ${\tilde G}_\frp>0$ (cf.~\Cref{remark:propG}). Therefore it holds
\bes
\begin{aligned}
\De \rho_\delta(\bar r_t^2)\leq \eqsp\frac{(\alpha_\frp- \tilde G_\frp)^+}{2}\eqsp \rho_\delta(\bar r_t^2)\De t\eqsp+2^{-1}| \nabla U_\frp-\nabla U_\frq|(Y_t)\De t\eqsp.
\end{aligned}
\ees
By taking expectation and integrating over time the above bound gives 
\bes
\begin{split}
\mathbb{E}[\rho_\delta(\bar r_t^2)]\leq \mathbb{E}[\rho_\delta(\bar r_0^2)]+\frac{(\alpha_\frp- \tilde G_\frp)^+}{2}\int_0^t \mathbb{E}[\rho_\delta(\bar r_s^2)]\eqsp\De s+2^{-1}\int_0^t\mathbb{E}[| \nabla U_\frp-\nabla U_\frq|(Y_s)]\eqsp\De s\\
=\sqrt{\delta}+\frac{(\alpha_\frp- \tilde G_\frp)^+}{2}\int_0^t \mathbb{E}[\rho_\delta(\bar r_s^2)]\eqsp\De s+\frac{t}{2}\,\int| \nabla U_\frp-\nabla U_\frq|\eqsp\De\frq\eqsp,
\end{split}
\ees
where in the last step we have relied on the fact that for any $t\geq 0$  $Y_t\sim \frq$ and that $\bar r_0=0$. Therefore Gronwall Lemma yields to
\bes
\begin{aligned}
\mathbb{E}[\bar r_t]\leq \mathbb{E}[\rho_\delta(\bar r_t^2)]
\leq \exp\biggl(\frac{t}{2}(\alpha_\frp- \tilde G_\frp)^+\biggr)\eqsp\biggl[\frac{t}{2}\int| \nabla U_\frp-\nabla U_\frq|\De\frq+\sqrt{\delta}\biggr].
\end{aligned}
\ees
By letting $\delta$ to zero in the above right-hand-side, we obtain the desired upper bound
\bes
\begin{aligned}
\bfW_{{f}}(\frq\,P_t^\frp ,\,\frq)\overset{   \eqref{eq:appo:equivalence:wasser:general}}{\leq} &\eqsp\tilde g_\frp{'}(0)\eqsp\bfW_1(\frq\,P_t^\frp ,\,\frq)\leq \tilde g_\frp{'}(0)\eqsp\mathbb{E}[|X^\frq_t-Y_t|]=\tilde g_\frp{'}(0)\eqsp\mathbb{E}[\bar r_t]\\
\leq&\eqsp \tilde g_\frp{'}(0)\eqsp\frac{t}{2}\eqsp\exp\biggl(\frac{t}{2}\eqsp(\alpha_\frp- \tilde G_\frp)^+\biggr)\eqsp\int| \nabla U_\frp-\nabla U_\frq|\eqsp\De\frq\eqsp.
\end{aligned}
\ees

\medskip

By putting together the last estimate with \eqref{eq:triangular:wasser:general} and \eqref{eq:appo:prima:wasser:general}  we have proven
\bes
\begin{split}
\bfW_{{f}}(\frp ,\,\frq)\leq &\,e^{-\lambda_\frp\, t}\eqsp \bfW_{{f}}(\frp ,\,\frq)+\tilde g_\frp{'}(0)\eqsp\frac{t}{2}\eqsp\exp\biggl(\frac{t}{2}\eqsp(\alpha_\frp- \tilde G_\frp)^+\biggr)\eqsp\int| \nabla U_\frp-\nabla U_\frq|\eqsp\De\frq\eqsp.
\end{split}
\ees
or equivalently
\bes
\bfW_{{f}}(\frp ,\,\frq)\leq\eqsp\frac{t/2}{1-e^{-\lambda_\frp\, t}}\eqsp\tilde g_\frp{'}(0)\eqsp\exp\biggl(\frac{t}{2}\eqsp(\alpha_\frp- \tilde G_\frp)^+\biggr)\eqsp\int| \nabla U_\frp-\nabla U_\frq|\eqsp\De\frq\eqsp,
\ees
which in the $t$ vanishing limit reads as
\bes
\bfW_{{f}}(\frp ,\,\frq)\leq\frac{\tilde g_\frp{'}(0)} {2\eqsp\lambda_\frp }\eqsp\int| \nabla U_\frp-\nabla U_\frq|\eqsp\De\frq\eqsp.
\ees
Combining the above bound with the equivalence \eqref{eq:appo:equivalence:wasser:general} concludes the proof.
\end{proof}

\begin{remark}[Explicit constants for $\tilde g_\frp$ as in \eqref{eq:tanh}]
Particularly, when $\tilde g_\frp$ is as in \eqref{eq:tanh} in \Cref{rem:g_L}, \ie 
\bes
\tilde g_\frp(r)=2\,(L)^{1/2}\,\tanh(r\,L^{1/2}/2)
\ees
for some $L$, then in the previous proof we have
$R=2(L)^{1/2}(L^{-1}+2/\alpha_\frp)$ and therefore 
\bes
    \bfW_1(\frp,\eqsp\frq)\leq \frac{\cosh^4(2(L)^{1/2}(L^{-1}+2/\alpha_\frp))}{\alpha_\frp+L}\eqsp\int|\nabla U_\frp-\nabla U_\frq|\De\frq\eqsp.
    \ees
\end{remark}

As a corollary of the above theorem, we are finally able to present the proof of \Cref{lemma:expl:sticky} where we give the key contraction estimates in $\bfW_1$ between conditional measures in the weak setting.

\begin{proof}[Proof of \Cref{lemma:expl:sticky}]
    Inequality \eqref{ineq:W1:psi} follows from the previous theorem when considering $\frp=\pi_T^{x,\psi^n}$ and $\frq=\pi^{x,\psistar}_T$. Indeed these two probabilities are the invariant measures associated to \eqref{SDE:p} with 
    $U_\frp(y)=(2T)^{-1}\,|y-x|^2+\psi^n(y)$ and $U_\frq(y)=(2T)^{-1}\,|y-x|^2+\psistar(y)$ respectively. \Cref{gio2022:thm} guarantees that there exists $\alpha_{\nu,n},\,\alpha_{\psistar}>-T^{-1}$ such that 
    \be
    \kappa_{U_\frp}(r)\geq T^{-1}+\alpha_{\nu,n}-r^{-1}{\tilde g}_\nu(r)\quad\text{ and }\quad
   \kappa_{U_\frq}(r)\geq T^{-1}+\alpha_{\psistar}-r^{-1}{\tilde g}_\nu(r)\eqsp.
   \ee
   Therefore $\frp$ and $\frq$ satisfy the assumptions of \Cref{teo:general}, which  gives \eqref{ineq:W1:psi}.

   We omit the details for the proof of \eqref{ineq:W1:phi} since it can be obtained in the same way, this time considering $\frp=\pi_T^{x,\varphi^{n+1}}$ and $\frq=\pi_T^{x,\varphistar}$.
\end{proof}

\section{Technical results and moment bounds}\label{technicalities}

\begin{lemma}[Exponential integrability of the marginals]\label{exp:integrability}
Let $\zeta \in \mcpalcd$ associated with $U:\rset^d\to \rset$ satisfying \eqref{eq:condition_a_log_concave_def}. Then for  any $\sigma \in(0,\alpha_U/2)$ it holds $\int \exp(\sigma |x|^2) \rmd \zeta (x) < \plusinfty$.
\end{lemma}
\begin{proof}
 It is enough noticing that for any $x\in\bbRD$ it holds
 \bes
 \begin{split}
 \langle \nabla U(x),\,x\rangle= \langle \nabla U(x)-\nabla U(0),\,x\rangle+\langle\nabla U(0),\,x\rangle\geq \kappa_{U}(|x|)\,|x|^2-|\nabla U(0)|\,|x|\\
 \geq \alpha_U\,|x|^2-({\tilde g}_U(|x|)+|\nabla U(0)|)|x|\geq \alpha_U\,|x|^2-\bar G_U\,|x|
 \end{split}
 \ees
    where above we have set $\bar G_U=\|{\tilde g}_U\|_{\infty}+|\nabla U(0)|$. 
  Therefore for any $x\in\bbRD$ it holds
    \bes\begin{split}
    U(x)=&U(0)+\int_0^1\langle\nabla U(tx), x\rangle\,\De t\\
    \geq& U(0)+\int_0^1(\alpha_U\,t\,|x|^2-\bar G_U\,|x|)\,\De t=U(0)+\frac{\alpha_U}{2}\,|x|^2-\bar G_U\,|x|\eqsp.\end{split}
    \ees

    Finally we may deduce for any $\sigma\in(0,\alpha_U/2)$
    \bes\begin{split}
    \|\exp(\sigma|x|^2)\|_{\mathrm{L}^1(U)}=\int_{\{|x|\leq R\}}\exp(\sigma|x|^2)\De\zeta(x)+\int_{\{|x|> R\}}\exp(\sigma|x|^2)\De\zeta(x)\\
    \leq \rme^{\sigma|R|^2}\,U(\{|x|\leq R\})+
    \rme^{-U(0)}\int_{\{|x|> R\}}  \exp( -(\alpha_U/2-\sigma)|x|^2+\bar G_U|x|)\,\De x\\
    \leq \rme^{\sigma|R|^2}\,U(\{|x|\leq R\})+
    \rme^{-U(0)}\int_{\{|x|> R\}}  \exp( -(\alpha_U/4-\sigma/2)|x|^2)\,\De x<+\infty\eqsp,\\
  \end{split}
  \ees
    where above we have set $
    R\coloneqq 2\bar G_U\,(\alpha_U/2-\sigma)^{-1}$.
\end{proof}

\subsection{Proof of  \eqref{eq:expl:grad}}\label{sec:proof-crefprop:cov}

\begin{proposition}\label{prop:cov}
Assume \Cref{A:entropy} holds. Then for any $n\in\nsets$,  $h\in\{\psistar,\,\psi^n,\varphistar,\varphi^n\}$ it holds
\be
\begin{aligned}
\nabla \log P_T\rme^{-h}(x) =&\, T^{-1}\int (y-x)\, \pi_T^{x,h}(\De y)\eqsp,\\
\nabla^2 \log P_T\rme^{-h}(x) =&\, -T^{-1}\operatorname{Id}+ T^{-2}\,\mathrm{Cov}(\pi_T^{x,h})\eqsp.
\end{aligned}
\ee
\end{proposition}
\begin{proof}
    We will only prove the case $h=\psi^n$ since the other cases can be proven with the same argument. The proof will run as in \cite[Proposition 5.2]{conforti2022quadratic}, once we have noticed that \Cref{A:entropy} and \Cref{exp:integrability} guarantees the validity of $\exp(\sigma_\nu |x|^2)\in \mathrm{L}^1(\nu)$ for some positive $\sigma_\nu>0$. Therefore from \eqref{eq:sinkhorn_iterate} we know that
    \bes
    \int_{\bbRD\times\bbRD} \exp\biggl(\sigma_\nu|y|^2-\varphi^n(x)-\psi^n(y)-\frac{|x-y|^2}{2T}\biggr)\,\De x\,\De y=\int_{\bbRD}\sigma_\nu|y|^2\,\De\nu(y)<+\infty\eqsp,
    \ees
    and hence there exists at least one point $\bar x\in\bbRD$ such that
    \bes
    \int_{\bbRD} \exp\biggl(\sigma_\nu|y|^2-\psi^n(y)-\frac{|\bar x -y|^2}{2T}\biggr)\,\De y<+\infty\eqsp.
    \ees
    Since for any $x\in\bbRD$ we can always write
    \bes
    |x-y|^2=|\bar x-y|^2-2\langle x-\bar x,y\rangle +|x|^2-|\bar x|^2\geq |\bar x-y|^2-2|x-\bar x|\,|y|+|x|^2-|\bar x|^2\eqsp,
    \ees
    for any $\bar \sigma<\sigma_\nu$ we have
    \be\label{appo:dominant}
    \int_{\bbRD} \exp\biggl(\bar \sigma|y|^2-\psi^n(y)-\frac{|\bar x -y|^2}{2T}\biggr)\,\De y<+\infty\quad\forall\,x\in\bbRD\eqsp.
    \ee
    
    This allows to differentiate under the integral sign in 
    \bes
    \log P_T\exp(-\psi^n)(x)=-\frac{d}{2}\log(2\pi T)+\log\int \exp\biggl(-\psi^n(y)-\frac{| x -y|^2}{2T}\biggr)\De y
    \ees
    and get the validity of \eqref{eq:expl:grad} for $h=\psi^n$, \ie
    \bes\begin{split}
    \nabla\log P_T\exp(-\psi^n)(x)&=-\frac{x}{T}+\frac1T\,\frac{\int y\exp(-\psi^n(y)-\frac{|x-y|^2}{2T})\De y}{\int\exp(-\psi^n(y)-\frac{|x-y|^2}{2T})\De y}\\
    &=T^{-1}\int (y-x)\,\pi^{x,\psi^n}_T(\De y)\eqsp.\end{split}
    \ees
    The bound \eqref{appo:dominant} guarantees to differentiate again the above integral and finally deduce our thesis.
\end{proof}

\end{appendix}




\end{document}